\documentclass[11pt,twoside, leqno]{article}

\usepackage{amsthm}
\usepackage{amssymb}
\usepackage{amsmath}
\usepackage{mathrsfs}
\usepackage{txfonts}
\usepackage{graphics}
\usepackage[Symbol]{upgreek}

\usepackage{extarrows}
\usepackage{graphics}
\usepackage{epsfig}
\usepackage{color}
\usepackage{verbatim}
\usepackage{hyperref}
\usepackage[mathlines,pagewise]{lineno}
\usepackage{amssymb,amsmath,amsfonts,amsthm,color,mathrsfs}
\usepackage[Symbol]{upgreek}
\usepackage{txfonts}
\usepackage[nottoc,notlot,notlof]{tocbibind}
\usepackage[active]{srcltx}
\usepackage{picinpar,graphicx} 
\usepackage{bbm}

\allowdisplaybreaks 
\def\rr{{\mathbb R}}
\def\rn{{{\rr}^n}}

\def\cn{{\mathbb N}}

\def\D{{\mathscr D}}

\def\ccc{{\mathcal C}}
\def\L{{\mathcal{L}}}
\def\LV{{\mathscr{L}}}
\def\HV{{\mathscr{H}}}
\def\PV{{\mathscr{P}}}

\def\fz{\infty}

\def\supp{{\mathop\mathrm{\,supp\,}}}

\def\loc{{\mathop\mathrm{\,loc\,}}}

\def\lz{\lambda}
\def\dz{\delta}

\def\gz{{\gamma}}

\def\vz{\varphi}

\def\r{\right}
\def\lf{\left}

\def\d{\mathrm{d}}
\def\e{\mathrm{e}}
\def\E{\mathscr{E}}
\def\BMO{{\mathrm{BMO}}}
\def\HMO{{\mathrm{HMO}}}

\def\eqref#1{(\ref{#1})}

\newtheorem{theorem}{Theorem}[section]
\newtheorem{lemma}[theorem]{Lemma}
\newtheorem{corollary}[theorem]{Corollary}
\newtheorem{proposition}[theorem]{Proposition}
\theoremstyle{definition}
\newtheorem{remark}[theorem]{Remark}
\newtheorem{definition}[theorem]{Definition}

\numberwithin{equation}{section}

\begin{document}
\arraycolsep=1pt
\author{Renjin Jiang} \arraycolsep=1pt
\arraycolsep=1pt

\title{\bf\Large On the Dirichlet problem for the
Schr\"odinger equation with boundary value in BMO space
\footnotetext{\hspace{-0.35cm}
2010 \emph{Mathematics Subject Classification}. 43A85, 42B35, 35J25.
\endgraf
{\it Key words and phrases}: Schr\"odinger equation, BMO, Carleson measure, metric measure space.
\endgraf
}}
\author{Renjin Jiang and Bo Li}
\date{}

\maketitle

\begin{center}
\begin{minipage}{11cm}\small
{\noindent{\bf Abstract.}
Let $(X,d,\mu)$ be a metric measure space satisfying a $Q$-doubling condition, $Q>1$, and
an $L^2$-Poincar\'{e} inequality. Let $\mathscr{L}=\mathcal{L}+V$ be a Schr\"odinger operator on $X$,
where $\mathcal{L}$ is a non-negative operator generalized by a Dirichlet form,
and $V$ is a non-negative Muckenhoupt weight    that satisfies a reverse H\"older condition $RH_q$ for some $q\ge (Q+1)/2$. We show that a solution to $(\mathscr{L}-\partial_t^2)u=0$ on $X\times \mathbb{R}_+$ satisfies the Carleson condition,
$$\sup_{B(x_B,r_B)}\frac{1}{\mu(B(x_B,r_B))}\int_{0}^{r_B}\int_{B(x_B,r_B)}|t\nabla u(x,t)|^2\frac{\d\mu\d t}{t}<\infty,$$
if and only if, $u$ can be represented as the Poisson
integral of the Schr\"odinger operator $\mathscr{L}$ with trace in the BMO space  associated with  $\mathscr{L}$.
}
\end{minipage}

\end{center}


\section{Introduction and main results}
\hskip\parindent
The study of harmonic extension of a function is one of the basic tools in harmonic analysis ever since
the seminar work of Stein and Weiss \cite{SW1960}. On the other hand, the Dirichlet problem
of elliptic equations is one of the basic problems in PDEs.
The harmonic extension of a BMO function plays a key role in Fefferman-Stein duality of
Hardy and BMO spaces; see \cite{FS1972} and also \cite{Fef1971}.
It was proved in \cite{FS1972}  that,
the harmonic function $u(x,t)=e^{-t\sqrt{-\Delta}}f(x)$ satisfies the following Carleson condition
\begin{align}\label{carleson}
\sup_{x_B,r_B} \frac{1}{r_B^n} \int_{0}^{r_{B}} \int_{B(x_B,r_B)} |t\nabla u(x, t)|^{2} \d x \frac{\d t}{t} \le C<\infty,
\end{align}
if and only if $f\in  {\BMO(\rn)}$,
where $\nabla=(\partial_t,\partial_{x_1},\cdots,\partial_{x_n})$ and $\Delta=\sum_{k=1}^n\partial^2_{x_k}$. Later on,
Fabes, Johnson and Neri \cite{FJN1976} found the Carleson condition \eqref{carleson}
actually characterize all harmonic functions $u(x,t)$ on $\mathbb{R}^{n+1}_+$ with boundary value
in ${\BMO(\rn)}$.
The study of this topic has widely extended to different settings including,
degenerate elliptic equations and systems,  elliptic equations and systems with complex coefficients,
also  Schr\"odinger equations, etc,
see \cite{ARR2015,Bar2015,DKP2011,DYZ2014,HLM2019,HL2018,HKMP2015,HLL2020,JXY2016,MMMM2019,STY2018} for instance.

Very recently, Duong, Yan and Zhang \cite{DYZ2014} extended non-trivially the study to the Schr\"odinger operator on $\rr^n$. More precisely, they proved that, if $L=-\Delta+V$, where the non-negative potential $V$
satisfies a reverse H\"older condition $RH_q$ for $q\ge n$,
then a solution $u$ to $-\partial_t^2 u+Lu=0$ on $\rr^{n+1}_+$ satisfies \eqref{carleson},
if and only if, $u$ can be represented as $u=e^{-t\sqrt L}f$, where $f$ is
in BMO space associated  to the operator $L$. As the solution $u$ is no longer smooth compared to
harmonic functions (cf. \cite{FJN1976}), their arguments are rather non-trivial.
Note that  recently developed theory on  Hardy and BMO spaces associated with operators (cf. \cite{DY2005-1,DY2005-2,DGMTZ2005,DzZi1999,DzZi2002,HLMMY2011,HMM2011,SY2016}),
and regularities to the Schr\"odinger equations (cf. \cite{Sh1999,Sh1995}),
play a role in the arguments of \cite{DYZ2014}.

Our main purpose is to extend the study of this topic to manifolds as well as general metric spaces,
where related Hardy and BMO spaces have been widely studied; see \cite{BDL18,Dz2005,HLMMY2011,SY2018,YZ2011}
for instance.
Note that the condition $V\in RH_q$ for $q\ge n$ assumed in \cite{DYZ2014} is to ensure
that one has a pointwise bound for the gradient of  heat kernel of the Schr\"odinger operator.
Even without the potential $V$, such bound
does not hold in general, unless assuming a strong
non-negative curvature condition or group structure (cf. \cite{Li1999,LinLiu2011}). Indeed, the pointwise bound of the gradient of heat kernel  fails already for uniformly elliptic operators; see \cite{acdh,AT1998,CJKS2020,JL2020} for instance.

To state our result, let us briefly describe our settings; see \cite{bd59,fot,GSC,St1996}
for more details.
Let $X$ be a  locally compact, separable, metrisable, and connected
space, $\mu$ be a  Borel measure that is finite on compact sets and strictly positive on non-empty open sets.
We consider {a strongly} local and regular Dirichlet form  $\E$ on $L^2(X,\mu)$ with dense domain $\D\subset L^2(X,\mu)$ (see \cite{fot,GSC}). We assume that $\E$  admits a ``{\it carr\'e du champ}",
which means that $\Gamma(f,g)$ is absolutely continuous with respect to $\mu$,
for all $f, g\in \D$. In what follows, we denote by  $\langle \nabla_x f,\nabla_x
g\rangle$  the energy density
$\frac{\d\Gamma(f,g)}{\d\mu}$, and by $|\nabla_x f|$  the square root
of $\frac{\d \Gamma(f,f)}{\d\mu}$.
We equip the space $(X,\mu,\E)$ with the intrinsic
(pseudo-)distance on $X$ associated to $\E$, defined by
$$d(x, y):=\sup\left\{f(x)-f(y):\, f\in\D_\loc\cap\ccc(X),\, |\nabla_x f|\le 1 \ \mbox{a.e.}\right\},$$
where $\ccc(X)$ denotes the class of all continuous functions on $X$.
We assume  that $d$ is  a true distance
 and that the topology induced by $d$ is equivalent to the original
topology on $X$.

The domain $\D$ endowed with the norm
$\sqrt{\|f\|^2_{2}+\E(f,f)}$ is a Hilbert space which we denote by $W^{1,2}(X)$. For an open set $U\subset X$,  the Sobolev space  $W^{1,p}(U)$ and  $W^{1,p}_0(U)$ is defined in a usual sense; see
\cite{BM1995,St1995,St1996} for instance.
Corresponding to such a Dirichlet form $\E$, there exists an operator  denoted by $\mathcal{L} $,
acting on a dense domain $\mathscr{D}(\mathcal{L})$ in
$L^2(X,\mu)$, $\mathscr{D}(\mathcal{L})\subset W^{1,2}(X)$, such that for all
$f\in \mathscr{D}(\mathcal{L})$ and each
$g\in W^{1,2}(X)$,
$$\int_X f(x)\mathcal{L} g(x)\d\mu(x)=\mathscr{E}(f,g).$$

Suppose that $0\le V\in RH_q(X)\cap A_\infty(X)$, where $RH_q(X)$ is the class of functions satisfying
reverse $q$-H\"older inequality and $A_\infty(X)$ is the class of all Muckenhoupt weights (cf. \cite{Mu1972,ST1989}); see Subsection 2.1
for more details. In the paper, we shall consider the Schr\"odinger operator
$$\mathscr L=\L+V.$$
Throughout this paper, we denote by $H_t=e^{-t\L}$ and $\mathscr{H}_t=e^{-t\LV}$ the heat semigroups,
$P_t=e^{-t\sqrt \L}$ and $\mathscr{P}_t=e^{-t\sqrt \LV}$ the Poisson semigroups, associated
with $\L$ and $\LV$, respectively.

Let $B(x,r)$ denote the open ball with center $x$ and radius $r$ with
respect to the distance $d$, and set $\lz B(x,r):=B(x,\lz r).$ We assume that $\mu$ is doubling, i.e.,
there exists a constant $C_d>0$ such that, for any ball $B=B(x,r)\subset X$,
$$\mu(2B(x,r))\le C_d\mu(B(x,r))<\fz.\leqno(D)$$
This implies that there exists $Q>1$ such that for any $0<r<R<\infty$ and $x\in X$,
$$\mu(B(x,R))\le C\left(\frac{R}{r}\right)^Q\mu(B(x,r)).$$

An $L^2$-Poincar\'{e} inequality is needed.
Precisely, we assume that $(X,d,\mu,\E)$ supports an $L^2$-Poincar\'{e} inequality, namely,
there exists a constant $C_P>0$ such that, for all balls $B=B(x,r)$, and $W^{1,2}(B)$ functions $f$,
$$
\lf(\fint_{B(x,r)}|f-f_B|^2\d\mu \r)^{1/2}\le C_Pr\lf(\fint_{B(x,r)}|\nabla_x f|^2\d\mu\r)^{1/2},\leqno(P_2)
$$
where and in what follows,
\begin{align*}
f_B=\fint_Bf\d\mu=\frac{1}{\mu(B)}\int_Bf\d\mu.
\end{align*}

We next recall the definition of $\LV_+$-harmonic functions on the upper half-space.
We say that $u\in W^{1,2}(X\times\rr_+)$ is an  $\LV_{+}$-harmonic function
in $X\times\rr_+$, if
$$\int_0^\fz\int_X \partial_t u\partial_t \phi \d \mu\d t+\int_0^\fz\int_X\langle \nabla_x u, \nabla_x \phi\rangle \d \mu\d t+\int_0^\fz\int_X Vu\phi \d \mu\d t=0$$
holds for all Lipschitz functions $\phi$ with compact support in $X\times\rr_+$.

The space $\HMO_\LV(X\times\rr_+)$ is defined as the class of all $\LV_+$-harmonic functions $u$,
that satisfies
$$\|u\|_{\HMO_\LV}:=\sup_{B(x_B,r_B)}\left(\int_0^{r_B}\fint_{B(x_B,r_B)}|t\nabla u(x,t)|^2\d\mu(x)\frac{\d t}{t}\right)^{1/2}<\infty.$$
We next introduce the critical function $\rho(x)$,
which was first introduced by Shen \cite[Definition 1.3]{Sh1995}, and has been
 studied in \cite{BDL18,DGMTZ2005,DzZi1999,YZ2011} for instance. For all $x\in X$, let
\begin{align}\label{crit-func-defn}
\rho(x):=\sup\lf\{r>0:\ \frac{r^2}{\mu(B(x,r))}\int_{B(x,r)}V\d\mu\le1\r\}.
\end{align}
Throughout the paper, we assume that $V\neq 0$ is a non-trivial potential,
which implies that $0<\rho(x)<\infty$ for any $x\in X$.
We say that a locally integrable function $f$ is in the space $\BMO_\LV(X)$ if
$$\|f\|_{\BMO_\LV}:=\sup_{{ B=B(x,r): \,r<\rho(x)}}\fint_{B}|f-f_B|\d\mu+\sup_{{ B=B(x,r):\,r\ge \rho(x)}}\fint_{B}|f|\d\mu <\infty.$$

Below is our main result.
\begin{theorem}\label{thm2}
Let $(X,d,\mu,\mathscr{E})$ be a complete Dirichlet metric space satisfying $(D)$
 with $Q>1$, and admitting an $L^2$-Poincar\'{e} inequality.
 Suppose $0\le V\in RH_{q}(X)\cap A_\infty(X)$ with $q\ge (Q+1)/2$, and $\LV=\L+V$.

 (i) If $u\in \HMO_{\LV}(X\times \rr_+)$, then there exists a function $f\in \BMO_\LV(X)$ such that $u(x,t)=\PV_tf(x)=e^{-t\sqrt\LV}f(x)$. Moreover there exists a constant $C>1$, independent of $u$, such that
$$\|f\|_{\BMO_\LV}\le C\|u\|_{\HMO_{\LV}}.$$

(ii) If $f\in \BMO_\LV(X)$, then $u(x,t)=\PV_tf(x)\in \HMO_{\LV}(X\times \rr_+)$,
and there exists  a constant $C>1$, independent of $f$, such that
$$\|u\|_{\HMO_{\LV}}\le C\|f\|_{\BMO_\LV}.$$
 \end{theorem}
For the classical Schr\"odinger operator $L=-\Delta+V$ on $\rr^n$,
Theorem \ref{thm2} improves the index $q\ge n$ in \cite{DYZ2014} to $q\ge (n+1)/2$.
Our  result  also applies to the uniformly elliptic operator, $L=-\mathrm{div} A\nabla+V$,
where $A=A(x)$ is an $n\times n$ matrix of real symmetric, bounded measurable coefficients, and satisfies the ellipticity condition, i.e.,  there exist constants $0<\lz\le\Lambda<\fz$ such that
$$\lz |\xi|^2\le \langle A \xi, \xi\rangle\le \Lambda |\xi|^2,\quad \forall\, \xi\in \rn.$$

The spaces $\HMO_L(\rr^{n+1}_+)$ and $\BMO_L(\rr^n)$ associated with $L$ are defined similarly to  $\BMO_\LV(X)$ and $\HMO_\LV(X\times\rr_+)$, with $X$ replaced by $\rr^n$.
Note that on $\rr^n$,
the reverse H\"older condition $V\in RH_{q}(\rr^n)$ implies that $V\in A_\infty(\rr^n)$ (cf. \cite{ST1989}).
\begin{corollary}\label{cor1.2}
Let $L=-\mathrm{div} A\nabla+V$, with $\mathrm{div} A\nabla$ being  a uniformly elliptic operator on $\rr^n$, $n\ge 2$, and $0\le V\in RH_{q}(\rr^n)$ for $q\ge (n+1)/2$.
Then  $u\in \HMO_L(\rr^{n+1}_+)$ if and only if there exists a function $f\in \BMO_L(\rr^n)$ such that $u(x,t)=e^{-t\sqrt L}f(x)$.
Moreover there exists a constant $C>1$ such that
$$C^{-1}\|f\|_{\BMO_L(\rr^n)}\le \|u\|_{\HMO_L(\rr^{n+1}_+)} \le C\|f\|_{\BMO_L(\rr^n)}.
$$
\end{corollary}
Our Theorem \ref{thm2} is applicable, if $L=-\mathrm{div} A\nabla$ satisfying certain degenerate  condition such as $A_2$-weight (cf. \cite{Mu1972,ST1989}), i.e.,
$$\lz w(x)|\xi|^2\le \langle A \xi, \xi\rangle\le \Lambda w(x)|\xi|^2,\quad \forall\, \xi\in \rn,$$
where $w\ge 0$ is an $A_2$-weight. The doubling condition is a consequence of the $A_2$ weight,
and it was known from \cite{FKS82} that an $L^2$-Poincar\'e
inequality holds.

Our result also applies to the Heisenberg group $\mathbb{H}^n$ and stratified  group.
On these spaces, a Fefferman-Stein decomposition  and a Carleson characterization of
the $\BMO_L(\mathbb{H}^n)$ space  were  established in \cite{LinLiu2011}.
Our result can be applied to more general manifolds, see \cite[Introduction]{acdh},
\cite[Section 7]{CJKS2020} and \cite{YZ2011} for more examples.

The paper is organized as follows. In Section 2, we provide basic materials
for the reverse H\"older class and Muckenhoupt weight, and properties of the critical
function $\rho$. We also provide some basic regularity results for solutions
to the Schr\"odinger equation. In Section 3, we provide estimates
for the heat and Poisson kernels of the Schr\"odinger operator,
and recall some properties for Hardy spaces associated with $\LV$.
The last two sections will be devoted to the proof of Theorem \ref{thm2}.

The letter $C$ (or $c$) will denote a positive constant that may vary from line to line but will remain independent of the main variables.

\section{Preliminaries and auxiliary tools}
\label{s2}
\hskip\parindent In this section,  we provide some basic properties of
the potential $V$, and regularity results for the Schr\"odinger equation.

\subsection{The reverse H\"{o}lder class and critical radii function}
\hskip\parindent Most of the properties that we need were proved in \cite[Section 1]{Sh1995}
 in the Euclidean space setting, and later extended to manifolds as well as metric spaces;
 see \cite{BDL18,YZ2011} for instance.

Let us first recall the reverse H\"{o}lder condition and Muckenhoupt weight.
\begin{definition}\label{RH}
\rm{(i)} A non-negative function $V$ on $X$  is said to be in the \textit{reverse H\"{o}lder $RH_q(X)$ class}
with $1< q\le\fz$,
if there exists a constant $C>0$ such that, for any ball $B\subset X$,
\begin{align*}
\lf(\fint_{B}V^q\d \mu\r)^{1/q}\le C\fint_{B}V\d \mu,
\end{align*}
with the usual modification when $q=\fz$.

(ii) A non-negative function $V$ on $X$ is said to be in the {\it{Muckenhoupt  weight class}} $A_\infty(X)$, if there exists some $1\le p<\fz$
such that
$$\sup_{B}\fint_{B}V\d\mu\left(\fint_B V^{\frac{1}{1-p}}\d\mu\right)^{p-1}\le C$$
when $p>1$, and when $p=1$
$$\sup_{B}\fint_{B}V\d\mu \left(\inf_{x\in B}V \right)^{-1}\le C.$$
Above the infimum is understood as the essential infimum.
\end{definition}

In the Euclidean setting $X=\rr^n$, it is well known that $RH_q(\rn)\subset A_\infty(\rn)$,
and every weight $V\in A_\infty(\rn)$ belongs to a reverse H\"older class $RH_p(\rn)$ for some
$p>1$. This is however not known in general metric measure space; see \cite[Chapter 1]{ST1989}.
Nevertheless, for a weight $V\in RH_q(X)\cap A_\infty(X)$, the induced measure $V\d\mu$ is also doubling (cf. \cite[Chapter 1]{ST1989}.
In what follows, we fix $C_V>0$ as the doubling constant of $V\d\mu$,
i.e.,
$$\int_{B(x,2r)}V\d\mu\le C_V\int_{B(x,r)}V\d\mu.\leqno(DV)$$

One remarkable feature about the $RH_q$ class is the self-improvement property, namely,
$V\in RH_q$ implies $V\in RH_{q+\epsilon}$ for some small $\epsilon>0$;
see \cite{ge73} for the Euclidean case and \cite[Chapter 3]{bb10} (see also \cite{ST1989}) for general cases.
This in particular implies $V\in L^{q'}_{\mathrm{loc}}(X)$ for some $q'$ strictly greater than $q$.
However, in general, the potential $V$ can be unbounded and does not belong to $L^p(X)$ for any $1\le p\le\fz$. As a model example, we could take $V(x)=|x|^2$ on $\rn$ with $n\ge3$.

By the definition of the critical function $\rho$ (see \eqref{crit-func-defn}),
there exist constants $C,c>0$ such that for any $x\in X$
\begin{align}\label{q5}
c\le\frac{\rho^2(x)}{\mu(B(x,\rho(x)))}\int_{B(x,\rho(x))}V\d\mu\le C;
\end{align}
see \cite[Subsection 2.1]{YZ2011} for instance. Moreover,
by \cite[Proposition 2.1 \& Lemma 2.1]{YZ2011}, we have the following property for the
critical function.
\begin{lemma}\label{crit-func-pro-1}
Let $V\in RH_q(X)\cap A_\infty(X)$, $q>\max\{Q/2,1\}$. Then there exist constants $C>0$ and $k_0\ge1$ such that
\begin{align*}
C^{-1}\rho(x)\lf(1+\frac{d(x,y)}{\rho(x)}\r)^{-k_0} \le \rho(y) \le C\rho(x)\lf(1+\frac{d(x,y)}{\rho(x)}\r)^{{k_0}/{(k_0+1)}},\quad\forall\, x,y\in X.
\end{align*}
\end{lemma}

\begin{lemma}\label{crit-func-pro-2}
Let $V\in RH_q(X)\cap A_\infty(X)$, $q>\max\{Q/2,1\}$.
There exist constants $C>0$ and $k_1=2+\log_2{C_V}$ such that, for any $x\in X$ and $r>0$,
\begin{eqnarray*}
\frac{r^2}{\mu(B(x,r))}\int_{B(x,r)}V\d\mu
\le
\lf\{\begin{array}{ll}
 \displaystyle{C\lf(\frac{r}{\rho(x)}\r)^{2-Q/q}},{\quad} &{r}<\rho(x);\\
\,\\
\displaystyle C\lf(\frac{r}{\rho(x)}\r)^{k_1},&{r}\ge\rho(x).
\end{array}\r.
\end{eqnarray*}
\end{lemma}
\begin{proof}
When ${r}<\rho(x)$, apply H\"{o}lder's inequality, the $RH_q$ condition, $(D)$ and \eqref{q5} to obtain
\begin{align*}
\frac{r^2}{\mu(B(x,r))}\int_{B(x,r)}V\d\mu
&\le r^2\lf[\frac{\mu(B(x,\rho(x)))}{\mu(B(x,r))}\r]^{1/q}\lf(\frac{1}{\mu(B(x,\rho(x)))}\int_{B(x,\rho(x))}V^q\d\mu\r)^{1/q} \\
&\le C\lf(\frac{r}{\rho(x)}\r)^2\lf[\frac{\mu(B(x,\rho(x)))}{\mu(B(x,r))}\r]^{1/q}\frac{\rho^2(x)}{\mu(B(x,\rho(x)))}\int_{B(x,\rho(x))}V\d\mu
 {\le C\lf(\frac{r}{\rho(x)}\r)^{2-Q/q}.}
\end{align*}

When $r\ge \rho(x)$, there exists some integer $j\ge0$ such that $2^j\rho(x)\le r<2^{j+1}\rho(x)$.
It follows from the doubling property of $V\d\mu$,  that
\begin{align*}
\frac{r^2}{\mu(B(x,r))}\int_{B(x,r)}V\d\mu
&\le \frac{(2^{j+1}\rho(x))^2}{\mu(B(x,2^j\rho(x)))}\int_{B(x,2^{j+1}\rho(x))}V\d\mu \\
&\le \frac{2^{2j+2}C_V^{j+1}\rho^2(x)}{\mu(B(x,\rho(x)))}\int_{B(x,\rho(x))}V\d\mu \\
&\le C2^{2j+2}C_V^{j+1}\le  C\lf(\frac{r}{\rho(x)}\r)^{k_1},
\end{align*}
where $k_1={2+\log_2{C_V}}$.
This concludes the proof of Lemma \ref{crit-func-pro-2}.
\end{proof}
\begin{remark}\rm
Note that on a connected metric measure space, a doubling measure also satisfies a reverse doubling property (cf. \cite[Remark 7.1.18]{hkst}), i.e., there exists $0<n<Q$ such that
$${ c2^n}\mu(B(x,r))\le \mu(B(x,2r)).$$
By the reverse doubling property of the measure $\mu$, we can actually obtain better estimate in the previous lemma for $r\ge \rho(x)$.
Indeed, one can take $k_1=2+\log_2C_V-n$. However, this
will not improve our main result. To keep the notions simple, we will not take the reverse doubling property into account.
\end{remark}
\begin{lemma}\label{crit-func-pro-3}
Let $V\in RH_q(X)\cap A_\infty(X)$, $q>\max\{Q/2,1\}$.
Then there exists a constant $C>0$ such that, for any $x\in X$ and $t>0$,
\begin{eqnarray*}
\int_X\frac{1}{\mu(B(x,\sqrt{t}))}\exp\lf\{-\frac{d(x,y)^2}{ct}\r\}V(y)\d\mu(y)\le
\lf\{\begin{array}{ll}
 \displaystyle\frac{C}{t}\lf(\frac{\sqrt{t}}{\rho(x)}\r)^{2-Q/q},{\quad} &\sqrt{t}<\rho(x);\\
\,\\
\displaystyle \frac{C}{t}\lf(\frac{\sqrt{t}}{\rho(x)}\r)^{k_1},&\sqrt{t}\ge\rho(x),
\end{array}\r.
\end{eqnarray*}
where $k_1=2+\log_2{C_V}$.
\end{lemma}
\begin{proof}
By dividing into annulus and using the fact that $V\d\mu$ is a doubling measure,  we conclude
\begin{align*}
&\int_X\frac{1}{\mu(B(x,\sqrt{t}))}\exp\lf\{-\frac{d(x,y)^2}{ct}\r\}V(y)\d\mu(y) \\
&\ {\le}\frac{C}{\mu(B(x,\sqrt{t}))}\int_X{\lf(1+\frac{d(x,y)}{\sqrt{t}}\r)^{-C_V}}V(y)\d\mu(y) \\
&\ {=}\frac{C}{\mu(B(x,\sqrt{t}))}\lf\{\int_{B(x,\sqrt{t})}+\sum_{j=0}^\fz\int_{2^j\le{\frac{d(y,x)}{\sqrt{t}}}<2^{j+1}}\r\}\cdots\d\mu(y) \\
&\ \le\frac{C}{\mu(B(x,\sqrt{t}))}\lf(\int_{B(x,\sqrt{t})}V\d\mu+\sum_{j=0}^\fz{2^{-jC_V}}\int_{2^{j+1}B(x,\sqrt{t})}V\d\mu\r) \\
&\ \le\frac{C}{\mu(B(x,\sqrt{t}))}\lf[\int_{B(x,\sqrt{t})}V\d\mu+\sum_{j=0}^\fz\lf(\frac{C_V}{2^{C_V}}\r)^j\int_{B(x,\sqrt{t})}V\d\mu\r] \\
&\ \le\frac{C}{\mu(B(x,\sqrt{t}))}\int_{B(x,\sqrt{t})}V\d\mu.
\end{align*}
Then the lemma follows from Lemma \ref{crit-func-pro-2}.
\end{proof}

\subsection{Some regularity results for the Schr\"odinger equation}
\hskip\parindent  In this part, we provide basic analysis tools for the Schr\"odinger operator.
We begin by a property of Muckenhoupt weight.

\begin{lemma}\label{lemma-small-weight}
 If $V\in A_\fz(X)$, then there exist $0 <\gamma,\epsilon< 1$ such that for all balls $B$ in $X$ we have
\begin{align*}
\mu\lf(\lf\{x\in B:\ V(x)< \gz \fint_BV\d\mu\r\}\r)\le\epsilon\mu(B).
\end{align*}
\end{lemma}
\begin{proof}
Fix a ball $B$ and take $$\lz=\exp\lf\{-\fint_B\log V\d\mu\r\}.$$
Then we have
\begin{align*}
\fint_B\log(\lz V)^{-1}\d\mu=-\fint_B\log(\lz V)\d\mu=-\log\lz-\fint_B\log V\d\mu=0.
\end{align*}
Since multiplication of an $A_\fz$ weight with a positive scalar does not alter its $A_\fz$ characteristic,
we may assume that $\fint_B\log V^{-1}\d\mu=0$.
Note that
\begin{align*}
[V]_{A_\fz}=\sup_B\lf(\fint_BV\d\mu\r)\exp\lf\{\fint_B\log V^{-1}\d\mu\r\},
\end{align*}
which gives $V_B\le [V]_{A_\fz}$.
Therefore by letting $\gz=[V]_{A_\fz}^{-1}(\e^{2[V]_{A_\fz}}-1)^{-1}$ and $\epsilon=1/2$,
 we obtain
\begin{align*}
\mu\lf(\lf\{x\in B:\ V(x)< \gz\fint_BV\d\mu\r\}\r)
&\le \mu\lf(\lf\{x\in B:\ V(x)< \gz [V]_{A_\fz}\r\}\r) \\
&\le \mu\lf(\lf\{x\in B:\ \log(1+V(x)^{-1})> \log(1+(\gz [V]_{A_\fz})^{-1})\r\}\r) \\
&\le \frac{1}{\log(1+(\gz [V]_{A_\fz})^{-1})}\int_B\log\frac{1+V(x)}{V(x)}\d\mu(x) \\
&\le \frac{1}{\log(1+(\gz [V]_{A_\fz})^{-1})}\int_B\log(1+V(x))\d\mu(x) \\
&\le \frac{1}{\log(1+(\gz [V]_{A_\fz})^{-1})}\int_BV(x)\d\mu(x) \\
&\le \frac{[V]_{A_\fz}}{\log(1+(\gz [V]_{A_\fz})^{-1})}\mu(B) \\
&=\frac12 \mu(B),
\end{align*}
which completes the proof.
\end{proof}

\begin{proposition}\label{FH} {\rm{(Fefferman-Phong's inequality)}}
Let $V\in RH_q(X)\cap A_\infty(X)$, $q>\max\{Q/2,1\}$. Suppose that $u$ belongs to $W^{1,2}(X)$.
Then there exists a constant $C>0$ independent of $u$ such that
\begin{align*}
\int_X|u|^2\rho^{-2}\d\mu\le C\lf(\int_X|\nabla_x u|^2\d\mu+\int_X|u|^2V\d\mu\r).
\end{align*}
\end{proposition}
\begin{proof}
Fix $x_0\in X$ and let $B_0=B(x_0,\rho(x_0))$.
There hold by the Poincar\'{e} inequality that
\begin{align*}
\fint_{B_0}\fint_{B_0}|u(x)-u(y)|^2\rho(x_0)^{-2}\d\mu(x)\d\mu(y)\le C\fint_{B_0}|\nabla_x u(x)|^2\d\mu(x)
\end{align*}
and
\begin{align*}
\fint_{B_0}\fint_{B_0}|u(y)|^2V(y)\d\mu(x)\d\mu(y)\le \fint_{B_0}|u(x)|^2V(x)\d\mu(x).
\end{align*}

Adding two inequalities above we arrive at
\begin{align*}
\fint_{B_0}|u(x)|^2\lf(\fint_{B_0}\min\lf\{V(y),\rho(x_0)^{-2}\r\}\d\mu(y)\r)\d\mu(x)\le C\lf(\fint_{B_0}|\nabla_x u(x)|^2\d\mu(x)+\fint_{B_0}|u(x)|^2V(x)\d\mu(x)\r).
\end{align*}

Since $V$ is an $A_\fz$ weight, Lemma \ref{lemma-small-weight}  shows there exist constants $0<\epsilon,\gamma<1$ such that, for any ball $B\subset X$,
\begin{align*}
\mu\lf(\lf\{y\in B:\ V(y)<\gamma \fint_BV\d\mu\r\}\r)\le \epsilon\mu(B).
\end{align*}
From this, \eqref{q5} and Lemma \ref{crit-func-pro-1}, it follows that
\begin{align*}
\fint_{B_0}\min\lf\{V(y),\rho(x_0)^{-2}\r\}\d\mu(y)\ge (1-\epsilon)\min\lf\{\gamma \fint_{B_0}V\d\mu,\rho(x_0)^{-2}\r\}\ge c\rho(x_0)^{-2}\ge c\rho(x)^{-2},
\end{align*}
which implies that
\begin{align*}
 \fint_{B_0}{|u(x)|^2}{\rho(x)}^{-2}\d\mu(x)\le C\lf(\fint_{B_0}|\nabla_x u(x)|^2\d\mu(x)+\fint_{B_0}|u(x)|^2V(x)\d\mu(x)\r).
\end{align*}
Therefore there holds by Lemma \ref{crit-func-pro-1} and $(D)$ that
\begin{align}\label{q19}
\int_X{|u(x)|^2}{\rho(x)^{-2}}\frac{\mathbbm{1}_{B_0}(x)}{\mu(B(x,\rho(x)))}\d\mu(x)
&\le C\int_X|\nabla_x u(x)|^2\frac{\mathbbm{1}_{B_0}(x)}{\mu(B(x,\rho(x)))}\d\mu(x) \\ \nonumber
&\ +C\int_X|u(x)|^2V(x)\frac{\mathbbm{1}_{B_0}(x)}{\mu(B(x,\rho(x)))}\d\mu(x).
\end{align}

Integrating the LHS of \eqref{q19} in $x_0$ over $M$, changing the order of integration,
applying Lemma \ref{crit-func-pro-1} and $(D)$ again, we infer that
\begin{align*}
\int_X\int_X{|u(x)|^2}{\rho(x)^{-2}}\frac{\mathbbm{1}_{B_0}(x)}{\mu(B(x,\rho(x)))}\d\mu(x)\d\mu(x_0)
&\ge \int_X\int_X{|u(x)|^2}{\rho(x)^{-2}}\frac{\mathbbm{1}_{B(x,c\rho(x))}(x_0)}{\mu(B(x,\rho(x)))}\d\mu(x_0)\d\mu(x) \\
&= \int_X{|u(x)|^2}{\rho(x)^{-2}}\frac{\mu({B(x,c\rho(x)))}}{\mu(B(x,\rho(x)))}\d\mu(x)\\
&\ge c\int_X{|u(x)|^2}{\rho(x)^{-2}}\d\mu(x).
\end{align*}
Analogously, the RHS is bounded by a harmless positive constant times
\begin{align*}
\int_X|\nabla_x u(x)|^2\d\mu(x)+\int_X|u(x)|^2V(x)\d\mu(x).
\end{align*}
This completes the proof of Proposition \ref{FH}.
\end{proof}

Recall that $\LV=\L+V$, where $\L$ is the operator generalized by the Dirichlet form $\E$.
We shall need the following  two Caccioppoli's inequalities, whose proofs are standard; see   \cite[Lemma 3]{Ku2000}
or \cite[Lemma 7.5]{BDL18} for example.
\begin{lemma}\label{Cacc}{\rm{(Caccioppoli's inequality)}} Let $V\in RH_q(X)\cap A_\infty(X)$, $q>\max\{Q/2,1\}$.
{ Let $u(x)$ be} a weak solution to $\LV u=\L u+Vu = 0$ in some ball $\lz B$ with $\lz>1$.
Then there exists a constant $C>0$ independent of $u$ and $\lz B$ such that
\begin{align*}
\int_{B}|\nabla_x u|^2\d\mu+\int_{B}|u|^2V\d\mu\le \frac{C}{(\lz r_B-r_B)^2}\int_{\lz B}|u|^2\d\mu.
\end{align*}
\end{lemma}

In what follows, we denote by $Q_0$, a parabolic cube as $Q_0=B(x_0,r)\times (t_0-r^2, t_0+r^2)$. Moreover, given a constant $\lambda>0$ we use $\lambda Q_0$ to denote
$B(x_0,\lz r)\times (t_0-\lz^2 r^2, t_0+\lz^2 r^2)$.

\begin{lemma}\label{Cacc1}{\rm{(parabolic Caccioppoli's inequality)}}
Let $V\in RH_q(X)\cap A_\infty(X)$, $q>\max\{Q/2,1\}$. Let $u(x,t)$ be a weak solution to $\partial_tu+\LV u = 0$ in some  parabolic cube  $\lz Q_0$ with $\lz>1$.
Then there exists a constant $C>0$ independent of $u$ and $\lz Q_0$ such that
\begin{align*}
\sup_{t_0-r^2<s<t_0+r^2}\int_{B(x_0,r)}|u(x,s)|^2\d\mu(x)+\int_{Q_0}(|\nabla_x u|^2+|u|^2V)\d\mu\d t\le \frac{C}{(\lz r-r)^2}\int_{\lz Q_0}|u|^2\d\mu\d t.
\end{align*}
\end{lemma}

\begin{proposition}\label{pointwise-bound}
Assume that the Dirichlet metric measure space $(X,d,\mu,\E)$ satisfies $(D)$ and
$(P_2)$.
Let $g\in L^q(B)$, where $q>\max\{\frac{2Q}{Q+2},1\}$ and $B=B(x,r)$.
There exists $v\in W^{1,2}_0(B)$ such that $\L v=g$.
Moreover, if $q<Q/2$, it holds for $q^\ast=\frac{Qq}{Q-2q}$,
$$\left(\fint_{\frac 12B}|v|^{q^\ast}\d\mu\right)^{1/q^\ast}\le Cr^2\left(\fint_{B}|g|^q\d\mu\right)^{1/q};$$
if $q=Q/2$ then it holds for any $p<\infty$ that
$$\left(\fint_{\frac 12B}|v|^{p}\d\mu\right)^{1/p}\le Cr^2\left(\fint_{B}|g|^q\d\mu\right)^{1/q};$$
if $q>Q/2$, it holds
$$\|v\|_{L^\infty(\frac 12B)}\le C  r^2\left(\fint_{B}|g|^q\d\mu\right)^{1/q}.$$
\end{proposition}
\begin{proof}
The existence of $v$ follows from \cite[Lemma 2.6]{CJKS2020}, which also implies
\begin{eqnarray}\label{existence-solution}
\fint_{B}|v|{\d}\mu\le Cr \lf(\fint_{B}{|\nabla_x v|^2}{\d}\mu\r)^{1/2}\le Cr^2 \lf(\fint_{B}|g|^{q}{\d}\mu\r)^{1/q}.
\end{eqnarray}
Note that \cite[Lemma 2.6]{CJKS2020} was stated for $Q\ge 2$ only. But when $Q<2$,
$\mu$ satisfies $Q'$-doubling with $Q'=2$, and $q>1=\frac{2Q'}{Q'+2}$, so \cite[Lemma 2.6]{CJKS2020}
applies in this case.

If $g\in L^\infty(B)$, \cite[Proposition 3.1]{CJKS2020} then implies for almost every $x\in \frac 12B$ that
$$|v(x)|\le C\lf\{\fint_{B}|v|{\d}\mu+G(x)\r\}\le Cr^2 \lf(\fint_{B}|g|^{q}{\d}\mu\r)^{1/q} +CG(x),$$
{where $G$} is given as
\begin{equation*}
G(x):=\sum_{j\le [\log_2r]}
2^{2j}\lf(\fint_{B(x,2^j)}|g|^{q}\d\mu\r)^{1/q}.
\end{equation*}
In the above summation, $[\log_2r]$ denotes the biggest integer not bigger
than $\log_2r$, and we set $g|_{B(x,2^j)\setminus B}=0$ if $B(x,2^j)\setminus B\neq \emptyset$.

By the mapping property of the potential $G$ from \cite[Theorem 5.3]{hak2000},
we deduce that  if $\frac{2Q}{Q+2}<q<Q/2$, it holds for $q^\ast=\frac{Qq}{Q-2q}$,
$$\left(\fint_{\frac 12B}|v|^{q^\ast}\d\mu\right)^{1/q^\ast}\le Cr^2 \lf(\fint_{B}|g|^{q}{\d}\mu\r)^{1/q} +C\|G\|_{L^{q^\ast}(\frac 12B)}\le Cr^2\left(\fint_{B}|g|^q\d\mu\right)^{1/q}.$$
{If $q=Q/2$, it holds for any $p<\infty$ that }
$$\left(\fint_{\frac 12B}|v|^{p}\d\mu\right)^{1/p}\le Cr^2\left(\fint_{B}|g|^q\d\mu\right)^{1/q}.$$
If $q>Q/2$, it holds
$$\|v\|_{L^\infty(\frac 12B)}\le Cr^2 \lf(\fint_{B}|g|^{q}{\d}\mu\r)^{1/q} +C\|G\|_{L^{\infty}(\frac 12B)}\le C  r^2\left(\fint_{B}|g|^q\d\mu\right)^{1/q}.$$

Since the Laplace equation is a linear equation, by using a limit progress, i.e., choosing $g_k=(-k)\vee\{g\wedge k\}$, we see that the above three estimates hold for general $L^q$ functions $g$.
\end{proof}

\begin{proposition}\label{infty-schrodinger}
Assume that the Dirichlet metric measure space $(X,d,\mu,\E)$ satisfies $(D)$ and
$(P_2)$.
 Suppose that $\LV u=\mathcal{L}u+Vu=0$ in a bounded domain $\Omega\subset X$, where $V\in RH_{q}(X)\cap A_\infty(X)$, $q>\max\{Q/2,1\}$. Then  $u$ is locally bounded in $\Omega$.
\end{proposition}
\begin{proof} Let $B=B(x_B,r_B)$ with $2B\subset \Omega$.
Suppose first $Q>2$.
Since $V\in RH_{q}$, $q>\max\{1,Q/2\} $,
and $u\in L^{\frac{2Q}{Q-2}}(2B)$ as a consequence of the Sobolev-Poincar\'e inequality (cf. \cite{hak2000}), by the H\"older inequality, we can conclude that
\begin{eqnarray}\label{regularity-improvement}
\int_{2B}|Vu|^p\d\mu\le \left(\int_{2B}|V|^{q}\d\mu\right)^{p/q}\left(\int_{2B}|u|^{\frac{pq}{q-p}}\d\mu\right)^{(q-p)/q},
\end{eqnarray}
where we choose $p<q$ such that $\frac{pq}{q-p}=\frac{2Q}{Q-2}$.
It then holds
$$\frac 1p =\frac 12-\frac 1Q +\frac 1q<\frac{1}{2}+\frac 1Q,$$
which implies $\frac{2Q}{Q+2}<p<q$.

Proposition \ref{pointwise-bound} implies that there exists
 a solution $v\in W^{1,2}_0(2B)$ such that $\L v=Vu$.
 As $\L v=Vu$ and $\LV u=0$ in $2B$, $u+v$ is $\L$-harmonic in $2B$ and it holds
\begin{equation}\label{MVP-harmonic}
\|u+v\|_{L^\infty(B)}\le C \fint_{2B}|u+v|\d\mu\le C \fint_{2B}|u|\d\mu+Cr_B^2\left(\fint_{2B}|Vu|^p\d\mu\right)^{1/p},
\end{equation}
where the last inequality follows from \eqref{existence-solution}.

If $p>Q/2$,  Proposition \ref{pointwise-bound}  implies
 $$\|v\|_{L^\infty(B)}\le C  r_B^2\left(\fint_{2B}|Vu|^p\d\mu\right)^{1/p},$$
 which together with \eqref{MVP-harmonic} yields
  $$\|u \|_{L^\infty(B)}\le C \fint_{2B}|u+v|\d\mu+\|v\|_{L^\infty(B)}\le C \fint_{2B}|u|\d\mu+C  r_B^2\left(\fint_{2B}|Vu|^p\d\mu\right)^{1/p}.$$
 This means $u$ is locally bounded in $\Omega$.

 If $p=Q/2$, then Proposition \ref{pointwise-bound} together with \eqref{MVP-harmonic} shows that
 for any $p'<\infty$, it holds
  $$\left(\fint_{B}|u|^{p'}\d\mu\right)^{1/p'}\le \|u+v\|_{L^\fz(B)}+\left(\fint_{B}|v|^{p'}\d\mu\right)^{1/p'}\le C \fint_{2B}|u|\d\mu+C  r_B^2\left(\fint_{2B}|Vu|^p\d\mu\right)^{1/p}.$$
As $q>Q/2$, we let $p'$ be large enough such that $\frac 1q+\frac 1{p'}<\frac 2Q$, and $p_1'$ be as
$\frac 1{p_1'} =\frac 1q+\frac 1{p'}.$
By the above inequality and an argument similar to \eqref{regularity-improvement}, we have $Vu\in L^{p_1'}(B)$, where $p_1'>Q/2$. The argument for the case $p>Q/2$ then applies and yields that
$u\in L^\infty(\frac 12B)$.

If $p<Q/2$, then Proposition \ref{pointwise-bound} together with \eqref{MVP-harmonic}
implies that for $p^\ast=\frac{Qp}{Q-2p}$,
  $$\left(\fint_{B}|u|^{p^\ast}\d\mu\right)^{1/p^\ast}\le C \fint_{2B}|u+v|\d\mu+\|v\|_{L^{p^\ast}(B)}\le C \fint_{2B}|u|\d\mu+C  r_B^2\left(\fint_{2B}|Vu|^p\d\mu\right)^{1/p}.$$
We let $p_1$ be such that $1/p_1=1/q+1/p^\ast$. Then it holds
$$\frac{1}{p_1}=\frac 1q+\frac1{p^\ast}=\frac 1q+\frac 1p-\frac 2Q<\frac{1}{p}. $$
By the above inequality and an argument similar to \eqref{regularity-improvement}, we have $Vu\in L^{p_1}(B)$. By using Proposition \ref{pointwise-bound} again, we first find a solution
 $v_1\in W^{1,2}_0(B)$ such that $\L v_1=Vu$ in $B$, and then repeat the above argument
 to deduce that $u\in L^\infty(\frac 12 B)$ if $p_1>Q/2$,
 $u\in L^{p_1^\ast}(\frac 12 B)$ if $p_1<Q/2$, and $u\in L^{p'}(\frac 12 B)$ for any $p'<\infty$ if $p_1=Q/2$.
 Let $k\in\cn$ be such that
 $$\frac 1{p_k}=\frac 1q+\frac 1{p_{k-1}^\ast} =\frac 1q+\frac 1{p_{k-1}}-\frac 2Q=\frac kq+\frac 1{p}-\frac {2k}Q<\frac 2Q. $$
 Repeating the above arguments at most $k$ times, we see that
 $u\in L^\infty(2^{-k}B)$. This implies that $u$ is locally bounded in $\Omega$.

 If $Q\le 2$, then $q>1$ and we may choose any $Q'>2$ such that $q>Q'/2$. Note that $\mu$ satisfies the doubling condition with $Q'$
 also. The above approach applies and shows that $u$ is locally bounded. The proof is complete.
 \end{proof}

 \begin{proposition}\label{MVP-Schrodinger}
 Assume that the Dirichlet metric measure space $(X,d,\mu,\E)$ satisfies $(D)$ and
$(P_2)$.
 Suppose that $\LV u=\mathcal{L}u+Vu=0$ in a bounded domain $\Omega\subset X$, where  $V\in RH_q(X)\cap A_\infty(X)$, $q>\max\{Q/2,1\}$.
 Then there exists $C>0$ such that for any ball $B=B(x_B,r_B)$ with $2B\subset \Omega$,
 $$\|u\|_{L^\infty(B)}\le C\fint_{2B}|u|\d\mu.$$
 Moreover, $u$ is locally H\"older continuous in $\Omega$, and there exists $\theta\in (0,\min\{1,2-Q/q\})$ such that for any $x,y\in \frac12 B$,
 $$|u(x)-u(y)|\le C{\left(\frac{d(x,y)}{r_B}\right)^\theta} \|u\|_{L^\infty(B)}\left(1+r_B^2\fint_B V\d\mu\right).$$
 \end{proposition}
\begin{proof}
By the previous proposition, $u$ is locally bounded, and hence $u^2\in W^{1,2}_\loc(\Omega)$.
For any ball $B=B(x_B,r_B)$ with $2B\subset \Omega$, let $\varphi\ge 0$ be a Lipschitz function
supported in $2B$, then it holds
\begin{eqnarray*}
\int_{2B}\langle \nabla_x u^2,  \nabla_x \varphi\rangle\d\mu&&=\int_{2B}\langle  \nabla_x  u , \nabla_x (2u\varphi)\rangle\d\mu- \int_{2B}2\varphi\langle  \nabla_x  u ,  \nabla_x u\rangle\d\mu\\
&&=-\int_{2B}2Vu^2\varphi\d\mu- \int_{2B}2\varphi| \nabla_x  u|^2\d\mu\le0.
\end{eqnarray*}
Therefore $u^2$ is a non-negative sub-harmonic function. \cite[Theorem 5.4]{BM1995} gives  that
 $$\|u\|_{L^\infty(B)}\le C\fint_{2B}|u|\d\mu.$$

Let us prove the H\"older continuity.
As $u\in L^\infty(B)$ and $V\in RH_q(X)\cap A_\infty(X)$, $q>\max\{1, Q/2\}$, by \cite[Theorem 5.13]{BM1995} (see also \cite[Lemma 2.8]{Jia2011})
the solution of the equation $\L v=Vu$, $v\in W^{1,2}_0(B)$  is H\"older continuous.
Moreover, there exits $\theta_1\in (0,\min\{1,2-{Q}/{q}\})$ such that for any $x,y\in \frac{1}{2}B$,
$$|v(x)-v(y)| \le \left(\frac{d(x,y)}{r_B}\right)^{\theta_1} \frac{r_B^2}{\mu(B)^{1/q}} \|Vu\|_{L^q(B)}.$$
Once more we use $u+v$ as an $\L$-harmonic function to conclude that there exits $\theta_2\in (0,1)$ such that for any $x,y\in \frac{1}{2}B$,
\begin{eqnarray*}
|u(x)-u(y)|&&\le |u(x)+v(x)- u(y)-v(y)|+|v(x)-v(y)| \\
&&\le C\left(\frac{d(x,y)}{r_B}\right)^{\theta_2} \fint_{B}|u+v|\d\mu+ C\left(\frac{d(x,y)}{r_B}\right)^{\theta_1} \frac{r_B^2}{\mu(B)^{1/q}} \|Vu\|_{L^q(B)}\\
&&\le  C\left(\frac{d(x,y)}{r_B}\right)^{\min\{\theta_1,\theta_2\}}\left( \fint_{B}|u|\d\mu+ r_B^2
\left(\fint_{B}|Vu|^q\d\mu\right)^{1/q}\right)\\
&&\le C\left(\frac{d(x,y)}{r_B}\right)^{\min\{\theta_1,\theta_2\}}\left( \fint_{B}|u|\d\mu+ \|u\|_{L^\infty(B)}r_B^2
\left(\fint_{B}|V|^q\d\mu\right)^{1/q}\right)\\
&&\le C\left(\frac{d(x,y)}{r_B}\right)^{\min\{\theta_1,\theta_2\}}\|u\|_{L^\infty(B)} \left( 1+r_B^2\fint_{B}V\d\mu\right).
\end{eqnarray*}
Letting ${\theta}=\min\{\theta_1,\theta_2\}$ completes the proof.
\end{proof}

\section{Heat kernel, Poisson kernel and Hardy space}\label{heat-poisson}

\subsection{Heat kernel and Poisson kernel}
\hskip\parindent We now proceed to estimate the heat kernel and Poisson kernel for the Schr\"odinger operator. Recall that $\L$ is the operator generalized by the Dirichlet form $\E$, $\LV=\L+V$ is the Schr\"odinger operator. We denote by
$h_t(x,y)$, $p_t(x,y)$ the kernels of $e^{-t\L}$ and $e^{-t\sqrt \L}$, respectively,
and $h^v_t(x,y)$, $p^v_t(x,y)$ the kernels of $\HV_t=e^{-t\LV}$ and $\PV_t=e^{-t\sqrt \LV}$, respectively.

We begin by a quantitative estimate for solutions to the parabolic Schr\"odinger equation.
The proof is similar to the  Euclidean case presented  in \cite[Proof of Theorem 2]{Ku2000}.

\begin{lemma}\label{Sh-2.1} Assume that $V\in RH_q(X)\cap A_\infty(X)$ with $q>\max\{Q/2,1\}$.
Let $u(x,t)$ be a non-negative weak solution to $(\partial_t+\LV)u=0$ in the
parabolic cube $4Q_0$, $Q_0=B(x_0,r)\times { (t_0-r^2, t_0+r^2)}$.
 Then there exist  constants $C,\epsilon>0$ independent of $u$ and $Q_0$ such that
\begin{align*}
\sup_{Q_0}|u| \le C\lf(\fint_{4Q_0}|u|^2\d\mu\d t\r)^{1/2}\exp\lf\{-\epsilon\lf(1+\frac{r}{\rho(x_0)}\r)^{1/{(k_0+1)}}\r\}.
\end{align*}
\end{lemma}
\begin{proof}
Since $u(x,t)$ is a non-negative weak solution to $(\partial_t+\LV)u=0$ and $V\ge 0$,
$u(x,t)$ is a non-negative sub-solution to $(\partial_t+\mathcal{L})u\le 0$.
The mean value property (see \cite[Theorem 2.1]{St1995}) implies that
\begin{align}\label{q1}
\sup_{Q_0}|u| \le C\lf(\fint_{2Q_0}|u|^2\d\mu\d t\r)^{1/2}.
\end{align}

For a given $k\in\cn$,  we divide the annulus $4Q_0\setminus 2Q_0$ into $2k$ equal shares.
For $j=0,1,\cdots,2k$,  we set $\alpha_{j}=2+j/k$. For each even number $1\le j\le 2k$,
we take a Lipschitz function $\eta_j$ with support in $\alpha_{j-1}B$ such that $\eta_j=1$ on $\alpha_{j-2}B$ and $|\nabla_x\eta_j|\le Ck/{r}$.

On the one hand, applying Fefferman-Phong's inequality (see Proposition \ref{FH}) yields that for each even number $2\le j\le 2k$,
\begin{eqnarray*}
\int_{\alpha_{j-2}Q_0}{|\eta_j u|^2}{\rho^{-2}}\d\mu\d t
&&\le C\int_{t_0-(\alpha_{j-1}r)^2}^{t_0+(\alpha_{j-1}r)^2}\lf(\int_X|u\nabla_x\eta_j|^2\d\mu+\int_X|\eta_j \nabla_x u|^2\d\mu+\int_X|\eta_j u|^2V\d\mu\r)\d t \nonumber\\
&&\le C\frac{k^2}{r^2}\int_{\alpha_{j-1}Q_0}|u|^2\d\mu\d t+C\int_{\alpha_{j-1}Q_0}(|\nabla_x u|^2+|u|^2V)\d\mu\d t \nonumber\\
&&\le C\frac{k^2}{r^2}\int_{\alpha_{j}Q_0}|u|^2\d\mu\d t,
\end{eqnarray*}
where the last inequality is due to parabolic Caccioppoli's inequality (Lemma \ref{Cacc1}).

On the other hand, thanks to Lemma \ref{crit-func-pro-1}, it follows that, for any $x\in 4B(x_0,r)$,
\begin{align*}
1+\frac{r}{\rho(x)}\ge1+\frac{cr}{\rho(x_0)}\lf(1+\frac{d(x,x_0)}{\rho(x_0)}\r)^{{-k_0}/{(k_0+1)}}\ge c\lf(1+\frac{r}{\rho(x_0)}\r)^{1/{(k_0+1)}}.
\end{align*}

One may use the two inequalities above to deduce that
\begin{align*}
\int_{\alpha_{j-2} Q_0}{|u|^2}\d\mu\d t
&\le C^2\lf(1+\frac{r}{\rho(x_0)}\r)^{{-2}/{(k_0+1)}}
\int_{\alpha_{j-2}Q_0}{|u(x,t)|^2}\lf(1+\frac{r}{\rho(x)}\r)^2\d\mu(x)\d t \\
&\le C^2\lf(1+\frac{r}{\rho(x_0)}\r)^{{-2}/{(k_0+1)}}
\int_{\alpha_{j-2}Q_0}{|u(x,t)|^2}\left[1+\lf(\frac{r}{\rho(x)}\r)^2\right]\d\mu(x)\d t\\
&\le C_1^2k^2\lf(1+\frac{r}{\rho(x_0)}\r)^{{-2}/{(k_0+1)}}\int_{\alpha_{j}Q_0}{|u|^2}\d\mu\d t,
\end{align*}
where $C_1>0$ is a harmless constant.

Combining the estimates for each $j=2,4,\cdots,2k$, we deduce  that
\begin{align*}
\int_{2Q_0}{|u|^2}\d\mu\d t
&\le   C_1^{2k}k^{2k}\lf(1+\frac{r}{\rho(x_0)}\r)^{{-2k}/{(k_0+1)}}\int_{4Q_0}{|u|^2}\d\mu\d t,
\end{align*}
which, together with \eqref{q1} and  $(D)$, implies that, for every $k=0,1,2,\cdots$,
\begin{align*}
\lf(1+\frac{r}{\rho(x_0)}\r)^{k/(k_0+1)}\sup_{Q_0}|u|
&\le C_2\lf(1+\frac{r}{\rho(x_0)}\r)^{k/(k_0+1)}\lf(\fint_{2Q_0}|u|^2\d\mu\d t\r)^{1/2} \nonumber\\
&\le C_2C_1^{k}k^{k}\lf(\fint_{4Q_0}|u|^2\d\mu\d t\r)^{1/2}.
\end{align*}
Now by Stirling's approximation, $k^k\sqrt{2\pi k}\sim e^{k}k! $ as $k\to \infty$, there holds
\begin{align}\label{q3}
\lf(1+\frac{r}{\rho(x_0)}\r)^{k/(k_0+1)}\sup_{Q_0}|u|\le  C\frac{C_1^{k}}{\sqrt {2\pi k}}e^{k}k!\lf(\fint_{4Q_0}|u|^2\d\mu\d t\r)^{1/2}.
\end{align}
Finally, multiplying both sides of \eqref{q3} by $\epsilon^{k}/{k!}$ for a positive constant $\epsilon$ small than $C_1e/2$, and summing over $k$, we arrive at
\begin{align*}
\sup_{Q_0}|u|\sum_{k=0}^\fz\frac{{\epsilon^{k}\lf(1+{r}/{\rho(x_0)}\r)^{k/(k_0+1)}}}{k!}\le C\lf(\fint_{4Q_0}|u|^2\d\mu\d t\r)^{1/2}\sum_{k=0}^\fz(C_1e\epsilon)^{k},
\end{align*}
and hence
\begin{align*}
\sup_{Q_0}|u|\le {C}\lf(\fint_{4Q_0}|u|^2\d\mu\d t\r)^{1/2} \exp\lf\{-\epsilon\lf(1+\frac{r}{\rho(x_0)}\r)^{{1}/{(k_0+1)}} \r\},
\end{align*}
as desired.
\end{proof}

By the estimates for heat kernel of heat semigroup of   operator $\L$ (cf. \cite{St1995,St1996}), we know that
\begin{align}\label{est-heat-laplace}
0\le  h^v_t(x,y)\le  h_t(x,y) \le \frac{C}{\mu(B(x, \sqrt{t}))} \exp \left\{-\frac{d(x, y)^{2}}{c t}\right\}.
\end{align}
Combining the results from \cite{St1996}, and Lemma \ref{Sh-2.1},
we derive the following estimates
for the heat kernel and Poisson kernel for the Schr\"odinger operators.
For the proofs, we shall make use of the results from \cite{Sal995,St1996}
and also \cite{BDL18}. Note that  although \cite[Section 7]{BDL18}
was presented on manifolds, its proof works also for our Dirichlet spaces.

\begin{proposition}\label{est-heat-kernel}
Let $(X,d,\mu,\mathscr{E})$ be a complete Dirichlet metric space satisfying $(D)$ and $(P_2)$.
Assume that $V\in RH_q(X)\cap A_\infty(X)$ with $q>\max\{Q/2,1\}$. Then the following statements hold true:

  (i) Gaussian upper bound: For every $m\in \{0\}\cup\cn$, there exist constants $C,c,\epsilon>0$ such that, for any $x,y\in X$ and  $t>0$,
$$
|t^{m}\partial^m_th^v_t(x,y)|\le \frac{C}{\mu(B(x, \sqrt{t}))} \exp \left\{-\frac{d(x, y)^2}{c t}\right\}\exp\lf\{-\epsilon\lf(1+\frac{\sqrt{t}}{\rho(x)}\r)^{{1}/{(k_0+1)}} \r\}.
\leqno(GUB)
$$

(ii) For every $m\in\cn$ and $N>0$, then there exists a constant $C>0$ such that, for any $x\in X$ and $t>0$,
\begin{align*}
   |t^m\partial^m_t{e^{-t\LV}}(1)(x)|\le C \lf(\dfrac{\sqrt{t}}{\rho(x)}\r)^{\dz}\lf(1+\dfrac{\sqrt{t}}{\rho(x)}\r)^{-N},
\end{align*}
{where $\delta\in(0,\min\{1,2-Q/q\})$.}

(iii) There exist constants $C,c>0$ and $\theta\in(0,\min\{1,2-Q/q\})$, such that
\begin{align*}
|h^v_t(x,y)-h^v_t(\overline{x},y)|\le C\lf(\frac{d(x,\overline{x})}{\sqrt{t}}\r)^{\theta}
 \frac{1}{\mu(B(x, \sqrt{t}))} \exp \left\{-\frac{d(x, y)^2}{c t}\right\},
\end{align*}
whenever $d(x, \overline{x}) < d(x, y)/4$, $d(x, \overline{x}) < \rho(x)$ and  {$t > 0$.}

(iv) There exist constants $C,c>0$ and $\theta\in(0,\min\{1,2-Q/q\})$, such that
\begin{align*}
|t\partial_th^v_t(x,y)-t\partial_th^v_t(\overline{x},y)|
\le C\lf(\frac{d(x,\overline{x})}{\sqrt{t}}\r)^{\theta}\frac{1}{\mu(B(x, \sqrt{t}))} \exp \left\{-\frac{d(x, y)^2}{c t}\right\},
\end{align*}
whenever $d(x, \overline{x}) < d(x, y)/4$, $d(x, \overline{x}) < \rho(x)$ and  $t > 0$.
\end{proposition}
\begin{proof}
(i)
It suffices to consider the case $m=0$ since the other case $m>0$ is a direct consequence of Cauchy's integral formula; see \cite[Theorem 2.6]{St1995} or \cite{da89} for instance.

When $m=0$, noticing that the heat kernel $h^v_t(x,y)$
is a weak solution to $\partial_tu+\LV u=0$ in parabolic cube $B(x,\sqrt{t}/2)\times (3t/4,5t/4)$,
we conclude by Lemma \ref{Sh-2.1}, \eqref{est-heat-laplace} and $(D)$ that
\begin{align*}
|h^v_t(x,y)|
&\le C\lf(\fint_{3t/4}^{5t/4}\fint_{B(x,\sqrt{t}/2)}|h^v_s(z,y)|^2\d\mu(z)\d s\r)^{1/2}\exp\lf\{-\epsilon\lf(1+\frac{\sqrt{t}}{\rho(x)}\r)^{{1}/{(k_0+1)}} \r\}\\
&\le C\lf(\fint_{3t/4}^{5t/4}\fint_{B(x,\sqrt{t}/2)}\frac{\d\mu(z)\d s}{[\mu(B(y,\sqrt{s}))]^2}\r)^{1/2}\exp\lf\{-\epsilon\lf(1+\frac{\sqrt{t}}{\rho(x)}\r)^{{1}/{(k_0+1)}} \r\}\\
&\le \frac{C}{\mu(B(y,\sqrt{t}))}\exp\lf\{-\epsilon\lf(1+\frac{\sqrt{t}}{\rho(x)}\r)^{{1}/{(k_0+1)}} \r\},
\end{align*}
which, together with \eqref{est-heat-laplace} and $(D)$ again, implies that
\begin{align*}
|h^v_t(x,y)|
&=|h^v_t(x,y)|^{1/2}|h^v_t(x,y)|^{1/2} \\
&\le \frac{C}{\mu(B(x, \sqrt{t}))^{1/2}\mu(B(y,\sqrt{t}))^{1/2}} \exp \left\{-\frac{d(x, y)^2}{c t}\right\}
 \exp\lf\{-\epsilon\lf(1+\frac{\sqrt{t}}{\rho(x)}\r)^{{1}/{(k_0+1)}} \r\} \\
&\le \frac{C}{\mu(B(x, \sqrt{t}))} \exp \left\{-\frac{d(x, y)^2}{c t}\right\}
 \exp\lf\{-\epsilon\lf(1+\frac{\sqrt{t}}{\rho(x)}\r)^{{1}/{(k_0+1)}} \r\}.
\end{align*}

(ii) Note that
\begin{align*}
\partial_te^{-t\LV}(1)(x)=-\LV e^{-t\LV}(1)(x)=-e^{-t\LV}V(x)=-\int_Xh^v_t(x,y)V(y)\d\mu(y).
\end{align*}
This, together with (i) and  Lemma \ref{crit-func-pro-3}, yields that
\begin{align*}
|t^m\partial^m_t e^{-t\LV}(1)(x)|
&\le C\int_X\frac{t}{\mu(B(x,\sqrt{t}))}\exp\lf\{-\frac{d(x,y)^2}{ct}\r\}\lf(1+\frac{\sqrt{t}}{\rho(x)}\r)^{-N-k_1-2}V(y)\d\mu(y) \\
&\le C {\lf(\dfrac{\sqrt{t}}{\rho(x)}\r)^{2-Q/q}}\lf(1+\dfrac{\sqrt{t}}{\rho(x)}\r)^{-N-2}
 \le C {\lf(\dfrac{\sqrt{t}}{\rho(x)}\r)^\dz}\lf(1+\dfrac{\sqrt{t}}{\rho(x)}\r)^{-N},
\end{align*}
where $k_1$ is as in Lemma \ref{crit-func-pro-3}, and $\delta\in(0,\min\{1,2-Q/q\})$.

(iii) By \cite[Propositon 7.15]{BDL18} together with the H\"older regularity of $h_t(x,y)$
from \cite{Sal995,St1996}, we know that
\begin{align*}
|h^v_t(x,y)-h^v_t(\overline{x},y)|
&\le |[h_t(x,y)-h^v_t(x,y)]-[h_t(\overline{x},y)-h^v_t(\overline{x},y)]|+|h_t(x,y)-h_t(\overline{x},y)| \\
&\le C{\lf(\frac{d(x,\overline{x})}{\sqrt{t}}\r)^\theta}\frac{1}{\mu(B(x, \sqrt{t}))} \exp \left\{-\frac{d(x, y)^2}{c t}\right\}.
\end{align*}

(iv) It follows from the semigroup property, $(D)$, (iii), and $(GUB)$ that
\begin{align*}
|t\partial_t h^v_t(x,y)-t\partial_t h^v_t(\overline{x},y)|
&=\lf|2\int_X[h^v_{t/2}(x,z)-h^v_{t/2}(\overline{x},z)]\lf.s\frac{\partial h^v_s(z,y)}{\partial s}\r|_{s=t/2}\d\mu(z)\r| \\
&\le C\int_X{\lf(\frac{d(x,\overline{x})}{\sqrt{t}}\r)^\theta}\frac{1}{\mu(B(x, \sqrt{t}))} \exp \left\{-\frac{d(x, z)^2}{c t}\right\}
 {\lf|\lf.s\frac{\partial h^v_s(z,y)}{\partial s}\r|_{s=t/2}\r|}\d\mu(z) \\
&\le C{\lf(\frac{d(x,\overline{x})}{\sqrt{t}}\r)^\theta}\frac{1}{\mu(B(x, \sqrt{t}))} \exp \left\{-\frac{d(x, y)^2}{c t}\right\}.
\end{align*}
The proof is completed.
\end{proof}

\begin{remark}\rm
In the estimate (iii) of the previous proposition, by more careful analysis, one can improve the estimate to
\begin{align*}
|h^v_t(x,y)-h^v_t(\overline{x},y)|\le C{\lf(\frac{d(x,\overline{x})}{\sqrt{t}}\r)^\theta}
 \frac{1}{\mu(B(x, \sqrt{t}))} \exp \left\{-\frac{d(x, y)^2}{c t}\right\}\lf(1+\frac{\sqrt{t}}{\rho(x)}+\frac{\sqrt{t}}{\rho(y)}\r)^{-N},
\end{align*}
whenever $d(x, \overline{x}) < \sqrt{t}$. Therefore, Proposition \ref{est-heat-kernel} (iv) can also be improved.
To keep the length of the paper, we will not address it.
\end{remark}

The following result was proved in \cite[Proposition 7.13]{BDL18}
by Kato-Trotter formula
$$h_t(x,y)-h^v_t(x,y)=\int_0^t h_s(x,z)V(z)h^v_{t-s}(z,y)\d\mu(z)\d s.$$
\begin{proposition}\label{heat-kernel-diff}
  Let $(X,d,\mu,\mathscr{E})$ be a complete Dirichlet metric space satisfying $(D)$ and $(P_2)$.
Let $V\in RH_q(X)\cap A_\infty(X)$, $q>\max\{Q/2,1\}$.
Then there exists $C,c>0$ such that for all $x,y\in X$ and $t>0$, it holds
\begin{equation*}
0\le h_t(x,y)-h^v_t(x,y)\le C\left(\frac{\sqrt t}{\sqrt t+\rho(x)}\right)^{2-Q/q} \frac{1}{\mu(B(x,\sqrt t))}\exp\left\{-\frac{d(x,y)^2}{ct}\right\}.
\end{equation*}
\end{proposition}

The  estimates for the Poisson kernel follow from the heat kernel estimates and the Bochner subordination formula,
\begin{align}\label{subordination}
e^{-t\sqrt \LV}
&= \frac{1}{2\sqrt{\pi}}\int_0^\fz\frac{t}{s^{1/2}}\exp\lf\{-\frac{t^2}{4s}\r\}e^{-s\LV}\frac{\d s}{s}
=\frac{1}{\sqrt{\pi}}\int_0^\fz \sqrt u e^{-u}e^{-\frac{t^2}{4u}\LV}\frac{\d u}{u}.
\end{align}
\begin{proposition}\label{est-poisson-kernel}
Let $(X,d,\mu,\mathscr{E})$ be a complete Dirichlet metric space satisfying $(D)$ and $(P_2)$.
Assume that $V\in RH_q(X)\cap A_\infty(X)$ with $q>\max\{Q/2,1\}$.
Then for any $N>0$, the following statements hold true:

  (i) Poisson upper bound: For every $m\in \{0\}\cup\cn$, there exists a constant $C>0$ such that, for any $x,y\in X$ and $t>0$,
\begin{eqnarray}\label{est-poisson-1}
|t^{m}\partial^m_tp^v_t(x,y)|\le C\frac{t}{t+d(x,y)}\frac{1}{\mu(B(x, t+d(x,y)))}\lf(1+\frac{t+d(x,y)}{\rho(x)}\r)^{-N}.
\end{eqnarray}

(ii) For every $m\in\cn$, then there exists a constant $C>0$ such that, for any $x\in X$ and $t>0$,
\begin{align}\label{est-poisson-3}
   |t^m\partial^m_te^{-t\sqrt{\LV}}(1)(x)|\le C \lf(\dfrac{{t}}{\rho(x)}\r)^\dz\lf(1+\frac{{t}}{\rho(x)}\r)^{-N},
\end{align}
{where $\delta\in(0,\min\{1,2-Q/q\})$.}

(iii) There exist constants $C,c>0$ and $\theta\in(0,\min\{1,2-Q/q\})$, such that
\begin{align*}
|p^v_t(x,y)-p^v_t(\overline{x},y)|
\le C{\lf(\frac{d(x,\overline{x})}{t+d(x,y)}\r)^\theta}\frac{t}{t+d(x,y)} \frac{1}{\mu(B(x, t+d(x,y)))},
\end{align*}
whenever $d(x, \overline{x}) < d(x, y)/4$, $d(x, \overline{x}) < \rho(x)$ and  $t > 0$.

(iv) There exist constants $C,c>0$ and $\theta\in(0,\min\{1,2-Q/q\})$, such that
\begin{align*}
|t\partial_tp^v_t(x,y)-t\partial_tp^v_t(\overline{x},y)|
\le C{\lf(\frac{d(x,\overline{x})}{t+d(x,y)}\r)^\theta\frac{t}{t+d(x,y)} \frac{1}{\mu(B(x, t+d(x,y)))},}
\end{align*}
whenever $d(x, \overline{x}) < d(x, y)/4$, $d(x, \overline{x}) < \rho(x)$ and  $t > 0$.
\end{proposition}
\begin{proof}
(i) When $m=0$, it follows from Bochner's subordination formula and $(GUB)$ that
\begin{align*}
|p^v_t(x,y)|
&= \lf|\frac{1}{2\sqrt{\pi}}\int_0^\fz\frac{t}{s^{1/2}}\exp\lf\{-\frac{t^2}{4s}\r\}{h^v_s(x,y)}\frac{\d s}{s}\r| \\ \nonumber
&\le C\int_0^\fz\frac{t}{s^{1/2}}\frac{1}{\mu(B(x,\sqrt{s}))}\exp\lf\{-\frac{t^2+d(x,y)^2}{cs}\r\}\lf(1+\frac{\sqrt{s}}{\rho(x)}\r)^{-N}\frac{\d s}{s}\\ \nonumber
&=C\lf\{\int_0^{t^2+d(x,y)^2}+\int^\fz_{t^2+d(x,y)^2}\r\}\cdots\frac{\d s}{s}.
\end{align*}
To estimate the local part, we conclude by the doubling condition that
\begin{align*}
&\int_0^{t^2+d(x,y)^2}\frac{t}{s^{1/2}}\frac{1}{\mu(B(x,\sqrt{s}))}\exp\lf\{-\frac{t^2+d(x,y)^2}{cs}\r\}\lf(1+\frac{\sqrt{s}}{\rho(x)}\r)^{-N}\frac{\d s}{s} \\
&\ = \int_0^1\frac{t}{\sqrt{r(t^2+d(x,y)^2)}}\frac{1}{\mu(B(x,{\sqrt{r(t^2+d(x,y)^2)}}))}
 \exp\lf\{-\frac{1}{cr}\r\}\lf(1+\frac{\sqrt{r(t^2+d(x,y)^2)}}{\rho(x)}\r)^{-N}\frac{\d r}{r} \\
&\ \le C\frac{t}{t+d(x,y)}\frac{1}{\mu(B(x,t+d(x,y)))}\lf(1+\frac{t+d(x,y)}{\rho(x)}\r)^{-N}\int_0^1{r^{-\frac{Q+N+1}{2}}}\exp\lf\{-\frac{1}{cr}\r\}\frac{\d r}{r} \\
&\ \le C\frac{t}{t+d(x,y)} \frac{1}{\mu(B(x, t+d(x,y)))}\lf(1+\frac{t+d(x,y)}{\rho(x)}\r)^{-N}.
\end{align*}
For the global part, there holds that
\begin{align*}
&\int^\fz_{t^2+d(x,y)^2}\frac{t}{s^{1/2}}\frac{1}{\mu(B(x,\sqrt{s}))}\exp\lf\{-\frac{t^2+d(x,y)^2}{cs}\r\}\lf(1+\frac{\sqrt{s}}{\rho(x)}\r)^{-N}\frac{\d s}{s} \\
&\ \le \frac{1}{\mu(B(x,\sqrt{t^2+d(x,y)^2}))}\lf(1+\frac{\sqrt{t^2+d(x,y)^2}}{\rho(x)}\r)^{-N}\int^\fz_{t^2+d(x,y)^2}\frac{t}{s^{1/2}}\frac{\d s}{s} \\
&\ \le C\frac{t}{t+d(x,y)} \frac{1}{\mu(B(x, t+d(x,y)))}\lf(1+\frac{{t+d(x,y)}}{\rho(x)}\r)^{-N}.
\end{align*}
Combing the three inequalities above leads to the desired result.

The case $m>0$ follows from the case $m=0$ and  Bochner's subordination formula that
\begin{align*}
|t^m\partial^m_tp^v_t(x,y)|
&=\left|\frac{1}{2\sqrt{\pi}}\int_0^\fz t^m\partial^m_t\left(\frac{t}{s^{1/2}}
\exp\lf\{-\frac{t^2}{4s}\r\}\right){h^v_s(x,y)}\frac{\d s}{s}\right|\\
&\le C_m\int_0^\fz \frac{t}{s^{1/2}}
\exp\lf\{-\frac{t^2}{cs}\r\}{h^v_s(x,y)}\frac{\d s}{s}\\
&\le  C\frac{t}{t+d(x,y)}\frac{1}{\mu(B(x, t+d(x,y)))}\lf(1+\frac{t+d(x,y)}{\rho(x)}\r)^{-N}.
\end{align*}

(ii) By Bochner's subordination formula \eqref{subordination} and Proposition \ref{est-heat-kernel} (ii), we deduce that
\begin{align*}
|t\partial_t{e^{-t\sqrt{\LV}}}(1)(x)|
&=\left|{\frac{t}{\sqrt{\pi}}}\int_0^\fz \sqrt u \frac{2t}{4u} e^{-u}\LV e^{-\frac{t^2}{4u}\LV}(1)(x)\frac{\d u}{u}\right|\\
&=\left|\frac{1}{\sqrt{\pi}}\int_0^\fz \frac{t}{\sqrt s} \exp\lf\{-\frac{t^2}{4s}\r\}s\LV e^{-s\LV}(1)(x)\frac{\d s}{s}\right|\\
&\le C\int_0^\fz\frac{t}{\sqrt s}\exp\lf\{-\frac{t^2}{4s}\r\}\lf(\frac{\sqrt s}{\rho(x)}\r)^{\dz}\lf(1+\frac{\sqrt s}{\rho(x)}\r)^{-N}\frac{\d s}{s},
\end{align*}
where $\delta\in(0,\min\{1,2-Q/q\})$.
Set $F:=\frac{t}{\sqrt s}\exp\lf\{-\frac{t^2}{4s}\r\}\lf(\frac{\sqrt s}{\rho(x)}\r)^{\dz}\lf(1+\frac{\sqrt s}{\rho(x)}\r)^{-N}$. Then we have
\begin{align*}
\int_{t^2}^\fz F\frac{\d s}{s}
&\le \int_{t^2}^\fz{ \frac{t}{\sqrt s}\lf(\frac{\sqrt s}{\rho(x)}\r)^{\dz}}\lf(1+\frac{t}{\rho(x)}\r)^{-N}\frac{\d s}{s}\le C\lf(\frac{t}{\rho(x)}\r)^{\dz}
\lf(1+\frac{t}{\rho(x)}\r)^{-N}.
\end{align*}
If $\rho(x)\ge t$, then
\begin{align*}
\int_0^{t^2}F\frac{\d s}{s}
&\le C\int_0^{t^2}\frac{t}{\sqrt s}\lf(\frac{s}{t^2}\r)^{2}\lf(\frac{\sqrt s}{\rho(x)}\r)^{\dz}\frac{\d s}{s}\le C\lf(\frac{t}{\rho(x)}\r)^{\dz}\le C\lf(\frac{t}{\rho(x)}\r)^{\dz}{\lf(1+\frac{t}{\rho(x)}\r)^{-N}.}
\end{align*}
Otherwise $\rho(x)<t$, there holds
\begin{align*}
\int_0^{t^2}F\frac{\d s}{s}
&=\int_0^{\rho(x)^2}F\frac{\d s}{s}+\int_{\rho(x)^2}^{t^2}F\frac{\d s}{s}\\
&\le C\int_0^{\rho(x)^2}\frac{t}{\sqrt s}\lf(\frac{\sqrt{s}}{t}\r)^{N+1}\lf(\frac{\sqrt s}{\rho(x)}\r)^{\dz}\frac{\d s}{s}
 +C\int_{\rho(x)^2}^{t^2}\frac{t}{\sqrt s}\lf(\frac{\sqrt{s}}{t}\r)^{N+1}\lf(\frac{\sqrt s}{\rho(x)}\r)^{\dz}\lf(\frac{\sqrt s}{\rho(x)}\r)^{-N}\frac{\d s}{s}\\
&\le C\lf(\frac{t}{\rho(x)}\r)^{\dz}\lf(\frac{\rho(x)}{t}\r)^{N}\le C\lf(\frac{t}{\rho(x)}\r)^{\dz}\lf(1+\frac{t}{\rho(x)}\r)^{-N}.
\end{align*}
Therefore we obtain
$$|t\partial_te^{-t\sqrt{\LV}}(1)(x)|\le C\lf(\frac{t}{\rho(x)}\r)^{\dz}\lf(1+\frac{t}{\rho(x)}\r)^{-N}.$$
The case $m\ge 2$ then follows similarly as
\begin{align*}
|t^m\partial^m_te^{-t\sqrt{\LV}}(1)(x)|
&=\left|\frac{t^m}{\sqrt{\pi}}\int_0^\fz \frac{1}{\sqrt s} \partial_t^{m-1}\left(\exp\lf\{-\frac{t^2}{4s}\r\}\right)s\LV e^{-s\LV}(1)(x)\frac{\d s}{s}\right|\\
&\le C\int_0^\fz\frac{t}{\sqrt s}{\exp\lf\{-\frac{t^2}{cs}\r\}}|s\LV e^{-s\LV}(1)(x)|\frac{\d s}{s}\\
&\le  C\lf(\frac{t}{\rho(x)}\r)^{\dz}\lf(1+\frac{t}{\rho(x)}\r)^{-N}.
\end{align*}

Propositions \ref{est-poisson-kernel} (iii) and (iv) follow
by using the Bochner subordination formula and Proposition \ref{est-heat-kernel} (iii), (iv).
We omit the proofs here and refer the reader to \cite[Proposition 3.6]{MSTZ2012} for an example.
\end{proof}
%
%

\subsection{Hardy space associated to Schr\"{o}dinger operator}
\hskip\parindent
Let us first recall the {definition} of Hardy  space for the Schr\"odinger operator; see \cite{BDL18,YZ2011}, and \cite{DY2005-1,Dz2005,DzZi1999,DGMTZ2005,DzZi2002,LinLiu2011}
for more related results.

An $L^2(X)$ function $a$ is called $(1,2)_\rho$-atom, if it satisfies

(i) $\supp a\subset B(x,r)$;

(ii) $\|a\|_{L^2(X)}\le \mu(B(x,r))^{-1/2}$;

(iii) $\int_X a\d\mu=0$  if $r<\rho(x)$.

A function $f$ is in the Hardy space $H^{1}_\rho(X)$, if $f$ has the representation $\sum_{j}\lambda_ja_j$, where $a_j$ are $(1,2)_\rho$-atoms, and $\sum_{j}|\lambda_j|<\infty$.
The  $H^1_\rho(X)$ norm is defined as
$$\|f\|_{H^1_\rho}:=\inf\left\{\sum_{j}|\lambda_j|:\,f=\sum_j\lambda_ja_j\right\}.$$

We shall need a characterization of the $H^1_\rho(X)$ via the following area function
based on the Poisson kernel.
For $f\in L^2(X)$, we define $\mathcal{S}_\PV(f)$ by
$$\mathcal{S}_\PV(f)(x)=\left(\int_0^\fz\fint_{B(x,{t})}
|t\sqrt{\LV}\PV_t(f)(z)|^2{\d \mu(z)}\frac{\d t}{t}\right)^{1/2}.$$
We define the $H^1_\LV(X)$ space norm as
 $\|f\|_{H^1_\LV}:=\left\|\mathcal{S}_\PV(f)\right\|_{L^1}.$
The Hardy space $H^1_\LV(X)$ is defined as the completion of all $L^2(X)$ functions
having finite $H^1_\LV(X)$ norm.
\begin{theorem}\label{hardy-equiv}
 Let $(X,d,\mu,\mathscr{E})$ be a complete Dirichlet metric space satisfying $(D)$ and $(P_2)$.
Suppose $V\in RH_q(X)\cap A_\infty(X)$ for some $q>\max\{1,Q/2\}$. Then the Hardy spaces
$H^1_\LV(X)$  and $H^1_\rho(X)$ coincide with equivalent norms.
Moreover, the dual space of $H^1_\LV(X)$ or $H^1_\rho(X)$ is the space $\BMO_\LV(X)$.
\end{theorem}
\begin{proof} By using the heat kernel estimates from Proposition \ref{est-heat-kernel},
the required result follows from a combination of
\cite[Theorem 7.3]{HLMMY2011}, \cite[Theorem 1.3]{SY2018}  and \cite[Theorem 2.12]{BDL18}.

The duality between $H^1_\rho(X)$ and $\BMO_\LV(X)$ was proved in \cite[Theorem 2.1]{YYZ2010-1}.
\end{proof}

As an application of the estimates of the Poisson kernel, we obtain that
\begin{proposition}\label{poisson-hardy}
 Let $(X,d,\mu,\mathscr{E})$ be a complete Dirichlet metric space satisfying $(D)$ and $(P_2)$.
Suppose $V\in RH_q(X)\cap A_\infty(X)$ for some $q>\max\{1,Q/2\}$.
For any $s>0$, $k\in\cn$ and $y\in X$, $\partial^k_s p^v_s(\cdot,y)\in H^1_\rho(X)$.
\end{proposition}
\begin{proof}
Note that {$-\sqrt{\LV}\PV_t(\partial^k_s p^v_s(\cdot,y))(x)=\partial_t^{k+1} p^v_{t+s}(x,y)$} is the kernel of
$\partial^{k+1}\PV_{s+t}$, and therefore by Proposition \ref{est-poisson-kernel}, we conclude that
\begin{align*}
 \mathcal{S}_\PV( \partial^k_s p^v_s(\cdot,y))(x)&=\left(\int_0^\fz\fint_{B(x,{t})}|t\sqrt{\LV}\PV_t(\partial^k_s p^v_s(\cdot,y))(z)|^2{\d \mu(z)}\frac{\d t}{t}\right)^{1/2}\\
 &\le \left(\int_0^\fz\fint_{B(x,{t})}|t\partial_t^{k+1} p^v_{t+s}(z,y)|^2{\d \mu(z)}\frac{\d t}{t}\right)^{1/2}\\
 &\le C\left(\int_0^\fz\fint_{B(x,{t})}\left(\frac{{t(t+s)}}{(t+s)^{k+1}(t+s+d(y,z))\mu(B(y,t+s+d(y,z)))}\right)^2{\d \mu(z)}\frac{\d t}{t}\right)^{1/2}\\
 &\le C\left(\int_0^\fz\fint_{B(x,{t})}{\left(\frac{t}{(t+s)^{k+1/2}(s+d(x,y))^{1/2}\mu(B(y,s+d(x,y)))}\right)^2}{\d \mu(z)}\frac{\d t}{t}\right)^{1/2}\\
 &\le C\left(\int_0^\fz{\left(\frac{t}{(t+s)^{k+1/2}(s+d(x,y))^{1/2}\mu(B(y,s+d(x,y)))}\right)^2}\frac{\d t}{t}\right)^{1/2}\\
 &\le {\frac{Cs^{1/2}}{s^{k}(s+d(y,x))^{1/2}\mu(B(y,s+d(y,x)))},}
\end{align*}
which implies for any $y\in X$,
$\mathcal{S}_\PV( \partial^k_s p^v_s(\cdot,y))(x)\in L^1(X)$.
Then by using Theorem \ref{hardy-equiv}, we conclude that  $\partial^k_s p^v_s(\cdot,y)\in H^1_\rho(X).$
\end{proof}

\section{From HMO to BMO}\label{s4}
\hskip\parindent  In this section, we show that for any $u\in \HMO_\LV(X\times \rr_+)$, there is
a $\BMO_\LV(X)$ function $f$ such that $u(x,t)=\PV_tf(x)$ with desired norm control.

We endow the product
space $X\times \rr_+$ (also $X\times \rr$) the product measure $\d\mu\d s$. 
Since $\d\mu$ is doubling,  $\d\mu\d s$ is also a doubling measure.
Moreover, as the Poincar\'e inequality holds on $(X,d,\mu)$, i.e.,
\begin{align*}
\lf(\fint_{B(x,r)}|f-f_B|^2\d\mu \r)^{1/2}\le C_Pr\lf(\fint_{B(x,r)}|{\nabla_xf}|^2\d\mu\r)^{1/2}
\end{align*}
for all $f\in W^{1,2}(B)$, we also have a Poincar\'e inequality for the product space as,
for any $g\in W^{1,2}(B(x,r)\times (s,s+r))$, it holds
\begin{align*}
\lf(\fint_{s}^{s+r}\fint_{B(x,r)}|g-g_{B(x,r)\times(s,s+r)}|^2\d\mu\d t \r)^{1/2}\le C_Pr\lf(\fint_{s}^{s+r}\fint_{B(x,r)}|\nabla g|^2\d\mu\d t\r)^{1/2},
\end{align*}
where $\nabla=(\nabla_x,\partial_t)$.

For $0\le V \in RH_q(X)\cap A_\infty(X)$ with $q>(Q+1)/2$, by defining $V(x,t):=V(x)$ for all $t\in\rr$,
we see that $V(x,t)\in RH_q(X\times\rr)\cap A_\infty(X\times \rr)$ with $q> (Q+1)/2$;
see the definitions in Section 2.
By Proposition \ref{MVP-Schrodinger}, $\LV_+$-harmonic functions are locally H\"older continuous on
$X\times \rr_+$, as $V \in RH_q(X\times\rr)\cap A_\infty(X\times\rr)$ with $q>(Q+1)/2$.

\begin{lemma}\label{lem3.1}
Let $(X,d,\mu,\mathscr{E})$ be a complete Dirichlet metric space satisfying $(D)$
and admitting an $L^2$-Poincar\'{e} inequality.
 Suppose $0\le V\in RH_{q}(X)\cap A_\infty(X)$ with $q>(Q+1)/2$.
If $u\in \HMO_\LV(X\times \rr_+)$, then    there exists a constant $C>0$ such that
  \begin{align*}
  |t\partial_t u(x,t)|\le C\|u\|_{\HMO_\LV}, \quad\forall\,x\in X\,\&\,t>0.
  \end{align*}
\end{lemma}
\begin{proof}
Let $\epsilon>0$. Given $x\in X$ and $t>\epsilon$, for any $-\epsilon<h<\epsilon$, we set
\begin{align*}
u_h(x,t)=\frac{u(x,t+h)-u(x,t)}{h}.
\end{align*}
For any Lipschitz function $\phi$ with compact support in $X\times (\epsilon,\infty)$, there holds that
\begin{align*}
&\int_\epsilon^\fz\int_X \partial_t u_h(x,t)\partial_t \phi (x,t) \d \mu(x)\d t+\int_\epsilon^\fz\int_X\langle \nabla_x u_h(x,t), \nabla_x \phi(x,t)\rangle \d \mu(x)\d t \\
&\ =\int_\epsilon^\fz\int_X \partial_t \frac{u(x,t+h)-u(x,t)}{h}\partial_t \phi (x,t) \d \mu(x)\d t \\
&\ \ +\int_\epsilon^\fz\int_X\langle \nabla_x \frac{u(x,t+h)-u(x,t)}{h},\nabla_x \phi(x,t)\rangle \d \mu(x)\d t\\
&\ =\frac{1}{h}\int_{h+\epsilon}^\fz\int_X \partial_t u(x,t)\partial_t {\phi (x,t-h)} \d \mu(x)\d t-\frac{1}{h}\int_\epsilon^\fz\int_X \partial_t u(x,t)\partial_t \phi (x,t) \d \mu(x)\d t\\
&\ \ +\frac{1}{h}\int_{h+\epsilon}^\fz\int_X\langle \nabla_x u(x,t),\nabla_x {\phi (x,t-h)}\rangle \d \mu(x)\d t-\frac{1}{h}\int_\epsilon^\fz\int_X\langle \nabla_x u(x,t), \nabla_x \phi(x,t)\rangle \d \mu(x)\d t \\
&\ =-\frac{1}{h}\int_{h+\epsilon}^\fz\int_X V(x)u(x,t){\phi (x,t-h)} \d \mu(x)\d t+\frac{1}{h}\int_\epsilon^\fz\int_XV(x)u(x,t)\phi(x,t)\d \mu(x)\d t\\
&\ =-\frac 1h\int_{\epsilon}^\fz\int_X V(x)u(x,t+h){\phi(x,t)} \d \mu(x)\d t+\frac{1}{h}\int_\epsilon^\fz\int_XV(x)u(x,t){\phi(x,t)}\d \mu(x)\d t\\
&\ =-\int_{\epsilon}^\fz\int_X V(x)u_h(x,t)\phi(x,t)\d \mu(x)\d t.
\end{align*}
This implies $u_h(x,t)$ is an $\LV_+$-harmonic function on $X\times (\epsilon,\infty)$.
Then we conclude by the mean value property (Proposition \ref{MVP-Schrodinger}) that
for any $t>2\epsilon$,
\begin{align}\label{f1}
 |u_h(x,t)|
\le C\lf(\fint_{B(x,t/2)}\fint^{3t/2}_{t/2}|u_h(y,s)|^2 {\d s}\d\mu(y)\r)^{1/2}.
\end{align}

Without loss of generality, we may assume $h>0$.
For a.e. $y\in B(x,t/2)$, by applying the H\"older inequality we obtain
\begin{align*}
\int^{3t/2}_{t/2}|u_h(y,s)|^2{\d s}
&=\int_{t/2}^{3t/2}\left(\frac{1}{h}\int_{s}^{s+h}\partial_r u(x,r)\d r\r)^2\d s\\
&\le\int_{t/2}^{3t/2}\frac{1}{h}\int_{s}^{s+h}|\partial_r u(x,r)|^2\d r\d s \\
&\le{\int_{t/2}^{3t/2+h}\frac{1}{h}\int_{t/2}^{3t/2}\mathbbm{1}_{\{s:\,s\in (r-h,r)\}}(s) \d s|\partial_r u(x,r)|^2\d r}   \\
&\le C\int_{t/2}^{3t/2+\epsilon}|\partial_r u(x,r)|^2\d r.
\end{align*}
Therefore, substituting this estimate into the RHS of \eqref{f1} yields for any $t>3\epsilon$ that
 \begin{align*}
 |u_h(x,t)|
&\le C\lf(\frac 1t\fint_{B(x,{t}/2)}\int_{{t}/2}^{3t/2+\epsilon}|\partial_su(y,s)|^2{\d s}\d \mu(y)\r)^{1/2} \\
&\le C\lf(\frac 1{t^2}{\fint_{B(x,2t)}}\int_{0}^{2t}|s\partial_su(y,s)|^2\frac{{\d s}}{s}\d \mu(y)\r)^{1/2} \\
& \le \frac{C}{t}\|u\|_{\HMO_\LV}.
\end{align*}
This implies for $t>3\epsilon$ that
$$|t\partial_tu(x,t)|\le {C}\|u\|_{\HMO_\LV}.$$
Letting $\epsilon\to 0$ shows the above
estimate holds for all $t>0$.
\end{proof}

\begin{proposition}\label{est-poisson-hmo}
Let $(X,d,\mu,\mathscr{E})$ be a complete Dirichlet metric space satisfying $(D)$
and admitting an $L^2$-Poincar\'{e} inequality.
 Suppose $0\le V\in RH_{q}(X)\cap A_\infty(X)$ with $q>(Q+1)/2$.
If $u\in \HMO_\LV(X\times \rr_+)$, then  there exists a constant $C>0$ such that for any $t,\epsilon>0$   and any $x\in X$,
$$\int_X\frac{|u(y,\epsilon)|^2}{(t+d(x,y))\mu(B(x,t+d(x,y)))}\d\mu(y)
<{\frac{C}{t}}\left(\|u(\cdot,\epsilon)\|^2_{L^\fz(B(x,2t))}+ \left(1+{\lf|\log\frac{t}{\epsilon}\r|}\right)^2\|u\|^2_{\HMO_\LV}\right).$$
\end{proposition}
\begin{proof}
By Proposition \ref{MVP-Schrodinger},
$u$ is locally H\"older continuous in $X\times\rr_+$ and locally bounded.
We spit the integral into $B(x,t)$ and $X\setminus B(x,t)$.
For the local part $B(x,t)$, one has
\begin{align*}
{\int_{B(x,t)}\frac{|u(y,\epsilon)|^2}{(t+d(x,y))\mu(B(x,t+d(x,y)))}\d \mu(y) \le\frac{1}{t}\|u(\cdot,\epsilon)\|^2_{L^\fz(B(x,t))}.}
\end{align*}

To estimate the global part $X\setminus B(x,t)$,  an annulus argument shows that
\begin{align*}
&\int_{X\setminus B(x,t)}\frac{|u(y,\epsilon)| ^2}{(t+d(x,y))\mu(B(x,t+d(x,y)))}{\d\mu(y)}\\
&\ \le C\sum_{j=1}^\fz (2^jt)^{-1}\fint_{2^{j-1}t}^{2^{j}t}\fint_{B(x,2^jt)\setminus  B(x,2^{j-1}t)}{|u(y,\epsilon)|^2}\d \mu(y)\d s \\
&\ \le C\sum_{j=1}^\fz (2^jt)^{-1} \fint_{E_j}{|u(y,\epsilon)-u(y,s)|^2}\d \mu(y)\d s\\
&\ \ +C\sum_{j=1}^\fz (2^jt)^{-1}\fint_{E_j}{|u(y,s)-u_{E_j}|^2}\d \mu(y)\d s+C\sum_{j=1}^\fz (2^jt)^{-1}{|u_{E_j}|^2}\\
&\ = C(I_1+I_2+I_3),
\end{align*}
where we denote by $E_j$ the cylinder  ${B(x,2^jt)\times[2^{j-1}t,2^{j}t)}$ for simplicity.

For the term $I_1$, there holds by Lemma \ref{lem3.1}  that
\begin{align*}
I_1
 &=\sum_{j=1}^\fz(2^jt)^{-1}\fint_{E_j}{\lf|{\int_\epsilon^s}\partial_ru(y,r)\d r\r|^2}\d \mu(y)\d s \\
&\le C\|u\|^2_{\HMO_\LV}\sum_{j=1}^\fz(2^jt)^{-1}\fint_{E_j}{\lf(\int_\epsilon^{2^{j}t}\frac{\d r}{r}\r)^2}\d \mu(y)\d s \\
&\le C \|u\|^2_{\HMO_\LV}\sum_{j=1}^\fz  \frac{\left|\log\frac{2^{j}t}{\epsilon} \right|^2}{2^{j}t} \le \frac C{t} \left(1+\left|\log\frac{t}{\epsilon}\right| \right)^2 \|u\|^2_{\HMO_\LV}.
\end{align*}

The estimate of term $I_2$ follows from the Poincar\'{e} inequality,
\begin{eqnarray*}
I_2&&\le \sum_{j=1}^\fz C (2^jt)^{-1} (2^{j}t)^2\fint_{E_j}{|\nabla u(y,s)|^2}\d \mu(y)\d s\\
&&\le \sum_{j=1}^\fz C (2^jt)^{-1}  \int_{0}^{2^jt}\fint_{B(x,2^jt)}{|s\nabla u(y,s)|^2}\d \mu(y)\frac{\d s}{s}\\
&&\le \frac{C}{t}\|u\|_{\HMO_\LV}^2.
\end{eqnarray*}
As $E_j= B(x,2^jt)\times[2^{j-1}t,2^{j}t)$, it holds
$E_j,E_{j+1}\subset  B(x,2^{j+1}t)\times [2^{j-1}t,2^{j+1}t)=:F_{j+1}$. Toward the term $I_3$, one writes
\begin{align*}
I_3
&\le \sum_{j=1}^\fz C(2^jt)^{-1} \lf(|u_{E_{1}}|+\sum_{i=2}^{j}|u_{E_{i}}-u_{E_{i-1}}|\r)^2\\
&\le \sum_{j=1}^\fz C(2^jt)^{-1} \lf(|(u-u(\cdot,\epsilon))_{E_{1}}|+\|u(\cdot,\epsilon)\|_{L^\infty(B(x,2t))}+\sum_{i=2}^{j}\left(|u_{E_{i}}-u_{F_i}|+|u_{F_i}-u_{E_{i-1}}|\right)\r)^2.
\end{align*}
It follows from the Poincar\'{e} inequality and the doubling property that
\begin{align*}
|u_{E_{i}}-u_{F_i}|+|u_{F_i}-u_{E_{i-1}}|
&\le C\lf(\fint_{2^{i-2}t}^{2^{i}t}\fint_{B(x,2^it)}|u(x,s)-u_{F_i}|^2\d \mu(x)\d s\r)^{1/2} \\
&\le C2^it\lf(\fint_{2^{i-2}t}^{2^{i}t}\fint_{B(x,2^it)}\frac{1}{s}|s\nabla u(x,s)|^2\d \mu(x)\frac{\d s}{s}\r)^{1/2}\le C \|u\|_{\HMO_\LV},
 \end{align*}
 and from Lemma \ref{lem3.1} that
\begin{align*}
|(u-u(\cdot,\epsilon))_{E_{1}}|&\le \fint_{B(x,t)\times [t,2t]}|u(y,s)-u(y,\epsilon)|\d\mu(y)\d s\\
& \le \fint_{B(x,t)\times [t,2t]}\lf(\int_\epsilon^s|\partial_r u(y,r)|\d r\r)\d\mu(y)\d s \le
C\left(1+\left|\log\frac{t}{\epsilon}\right| \right)\|u\|_{\HMO_\LV}.
\end{align*}
The above two estimates yield
\begin{align*}
I_3
&\le  \frac Ct\sum_{j=1}^\fz{2^{-j}}\lf(\|u(\cdot,\epsilon)\|_{L^\fz(B(x,2t))}+j\left(1+\left|\log\frac{t}{\epsilon}\right| \right)\|u\|_{\HMO_\LV} \r)^2\\
&\le  \frac Ct \lf(\|u(\cdot,\epsilon)\|^2_{L^\fz(B(x,2t))}+\left(1+\left|\log\frac{t}{\epsilon}\right| \right)^2 \|u\|^2_{\HMO_\LV}\r).
 \end{align*}
Combining the estimates of $I_1$, $I_2$ and $I_3$ leads to the required conclusion.
\end{proof}

 \begin{corollary}\label{lem3.2}
 Let $(X,d,\mu,\mathscr{E})$ be a complete Dirichlet metric space satisfying $(D)$
and admitting an $L^2$-Poincar\'{e} inequality.
 Suppose $0\le V\in RH_{q}(X)\cap A_\infty(X)$ with $q>(Q+1)/2$.
If $u\in \HMO_\LV(X\times \rr_+)$,  then  for  any $t,\epsilon>0$,
the Poisson extension $\PV_t(u(\cdot,\epsilon))$ is well defined.
\end{corollary}
\begin{proof}
By Proposition \ref{est-poisson-kernel}, the Poisson kernel satisfies
  $$
|p^v_t(x,y)|\le C\frac{t}{t+d(x,y)}\frac{1}{\mu(B(x, t+d(x,y)))}\lf(1+\frac{t+d(x,y)}{\rho(x)}\r)^{-N}.
$$
By this, the H\"older inequality and Proposition \ref{est-poisson-hmo}, we conclude
\begin{eqnarray*}
|\PV_t(u(\cdot,s))(x)|&&\le C|\PV_t(1)(x)|^{1/2}\left(\int_X\frac{t|u(y,s)|^2}{(t+d(x,y))\mu(B(x,t+d(x,y)))}\d\mu(y)\right)^{1/2}\nonumber\\
&&{ \le C\left(\|u(\cdot,s)\|_{L^\fz(B(x,2t))}+\left(1+\left|{\log\frac{t}{s}}\right| \right)\|u\|_{\HMO_\LV}\right)<\infty,}
\end{eqnarray*}
as desired.
\end{proof}

\begin{proposition}\label{Liouville}
Let $(X,d,\mu,\mathscr{E})$ be a complete Dirichlet metric space satisfying $(D)$
and admitting an $L^2$-Poincar\'{e} inequality.
 Assume $0\le V\in RH_{q}(X)\cap A_\infty(X)$ with $q>\max\{1,Q/2\}$.
Suppose that $w$ is a solution to $\LV w=\mathcal{L}w+Vw=0$ on $X $.
If there exists $n>0$ such that
\begin{equation}\label{global-control}
\int_X \frac{|w(y)|^2}{(1+d(x,y))^n\mu(B(x,1+d(x,y)))}{\d\mu(y)}<\infty,
\end{equation}
then $w\equiv 0$.
\end{proposition}
\begin{proof}
Note that Proposition \ref{MVP-Schrodinger} implies $w$ is locally bounded. If \eqref{global-control} holds for some $x\in X$, then it holds for all $z\in X$ that
\begin{align*}
&\int_X \frac{|w(y)|^2}{(1+d(z,y))^n\mu(B(z,1+d(z,y)))}\d\mu(y)\\
&\ \le {\int_{X\setminus B(z,2d(x,z))}} \frac{|w(y)|^2}{(1+d(z,y))^n\mu(B(z,1+d(z,y)))}\d\mu(y)+C\|w\|_{L^2(B(z,2d(x,z)))}^2\\
&\ \le C{\int_{X\setminus B(z,2d(x,z))}} \frac{C|w(y)|^2}{(1+d(x,y))^n\mu(B(x,1+d(x,y)))}\d\mu(y)+C\|w\|_{L^2(B(z,2d(x,z)))}^2\\
&\ <\infty.
\end{align*}

For any $x_B\in X$, let $B:=B(x_B,r_B)$ with $r_B>>\max\{\rho(x_B),1\}$.
Let $\varphi$ be a Lipschitz function on $X $, that satisfy
$\varphi=1$ on $4B$, $\supp  \varphi \subset
5B$ and $|\nabla_x \varphi|\le C/r_B$.
Note that $w$ is locally H\"older continuous by Proposition \ref{MVP-Schrodinger}.
For any $x\in B(x_B,\rho(x_B))$, it follows from $\LV w=0$ that
\begin{align*}
w(x)\varphi(x)
&={-{\int_0^\infty \partial_t e^{-t \LV}( w\varphi)(x)\d t}}{=\int_0^\infty \LV e^{-t \LV}( w\varphi)(x)\d t} \\
&=\int_0^\infty \int_{X}{V(y)} h^v_t(x,y) (w\varphi)(y)\d\mu(y)\d t +\int_0^\infty \int_{X} \langle \nabla_yh^v_t(x,y),\nabla_y (w\varphi)(y)\rangle\d\mu(y)\d t\\
&=\int_0^\infty \int_{X}{V(y)} h^v_t(x,y) (w\varphi)(y)\d\mu(y)\d t +\int_0^\infty \int_{X} \langle \nabla_y (h^v_t(x,y)\varphi(y)),\nabla_y w(y)\rangle\d\mu(y)\d t\\
&\ -\int_0^\infty \int_{X} \left[h^v_t(x,y) \langle \nabla_y \varphi(y),\nabla_y w(y)\rangle-w(y)\langle \nabla_yh^v_t(x,y), \nabla_y \varphi(y)\rangle \right]\d\mu(y)\d t\\
&=-\int_0^\infty \int_{X} \left[h^v_t(x,y) \langle \nabla_y \varphi(y),\nabla_y w(y)\rangle -
w(y)\langle \nabla_yh^v_t(x,y), \nabla_y \varphi(y)\rangle \right]
\d\mu(y)\d t.
\end{align*}
This implies
\begin{eqnarray*}
|w(x)\varphi(x)|&&\le \frac{C}{r_B}\int_0^\infty \int_{5B\setminus 4B} \left[h^v_t(x,y)|\nabla_y w(y)|+
| w(y)||\nabla_yh^v_t(x,y)|\right]
\d\mu(y)\d t.
\end{eqnarray*}
Note that by applying Caccioppoli's inequality (Lemma \ref{Cacc}) to  $w$ on $X$,
$$ \left(\fint_{5B}|\nabla_y w(y)|^2\d\mu(y)\right)^{1/2}\le \frac{C}{r_B}\left(\fint_{6B}|w(y)|^2\d\mu(y)\right)^{1/2}.$$
Similarly, by letting
$\phi$ be a Lipschitz function on $X$, that satisfies
$\phi=1$ on $5B\setminus 4B$, $\supp  \phi \subset
6B\setminus 3B$ and $|\nabla_y\phi|\le \frac{C}{r_B}$,
we deduce
\begin{align*}
&\int_{6B}|\nabla_y h^v_t(x,y)|^2\phi^2(y)\d\mu(y)\\
&\  =\int_{6B}\langle \nabla_yh^v_t(x,y),\nabla_y(h^v_t(x,y)\phi^2(y))\rangle\d\mu(y)- \int_{6B}h^v_t(x,y) \langle \nabla_yh^v_t(x,y),\nabla_y(\phi^2(y))\rangle\d\mu(y)\\
&\ =-\int_{6B}V(y){|h^v_t(x,y)|^2}\phi^2(y)\d\mu(y)-\int_{6B}\langle \partial_t h^v_t(x,y),h^v_t(x,y)\phi^2(y)\rangle\d\mu(y)\\
&\ \ -\int_{6B}h^v_t(x,y) \langle \nabla_yh^v_t(x,y),\nabla_y(\phi^2(y))\rangle\d\mu(y)\\
&\ \le \int_{6B}|\partial_t h^v_t(x,y)| h^v_t(x,y)\phi^2(y) \d\mu(y)+\int_{6B} \left[2|\nabla_y\phi(y)|^2{|h^v_t(x,y)|^2} + \frac 12|\nabla_yh^v_t(x,y)|^2 \phi^2(y)\right]\d\mu(y).
\end{align*}
Consequently, it follows
\begin{eqnarray*}
\int_{5B\setminus 4B}|\nabla_y h^v_t(x,y)|^2{\d\mu(y)}
&& \le C\int_{6B\setminus 3 B}\left[|\partial_t h^v_t(x,y)| h^v_t(x,y) +\frac{1}{r_B^2}{|h^v_t(x,y)|^2}\right]{\d\mu(y)}\\
&&\le C\int_{6B\setminus 3 B}\left[r_B^2|\partial_t h^v_t(x,y)|^2 +\frac{1}{r_B^2}{|h^v_t(x,y)|^2}\right]{\d\mu(y)}.
\end{eqnarray*}
By the above two Caccioppoli's inequalities together with the H\"older inequality, we obtain for any $x\in B(x_B,\rho(x_B))$,
\begin{align*}
|w(x)|
&\le \frac{C}{r_B}\int_0^\infty \int_{5B\setminus 4B} \left[h^v_t(x,y))|\nabla_y w(y)|+| w(y)||\nabla_yh^v_t(x,y)|\right]\d\mu(y)\d t\\
&\le \frac{C}{r_B} \int_0^\infty \left[ \left(\int_{5B\setminus 4B}{|h^v_t(x,y)|^2}\d\mu(y)\right)^{1/2}
\left(\int_{5B\setminus 4B}|\nabla_yw(y)|^2\d\mu(y)\right)^{1/2}\right. \\
&\ \left.+\left(\int_{5B\setminus 4B}|\nabla_yh^v_t(x,y)|^2\d\mu(y)\right)^{1/2}\left(\int_{5B\setminus 4B}{|w(y)|^2}\d\mu(y)\right)^{1/2}\right]\d t\\
&\le \frac{C}{r_B}\left(\int_{6B}{|w(y)|^2}\d\mu(y)\right)^{1/2}
\int_0^\infty \left[\int_{6B\setminus 3B} \left( \frac{1}{r_B^2}{|h^v_t(x,y)|^2}+r_B^2|\partial_t h^v_t(x,y)|^2\right)\d\mu(y)\right]^{1/2}\d t.
\end{align*}

Recall that $r_B>>\max\{\rho(x_B),1\}$ and $\rho(x)\sim \rho(x_B)$ for any $x\in B(x_B,\rho(x_B))$.
By using
Proposition \ref{est-heat-kernel}, we conclude that for $x\in B(x_B,\rho(x_B))$,
\begin{align*}
&\int_0^\infty \left(\int_{6B\setminus 3B}\frac{1}{r^2_B}{|h^v_t(x,y)|^2}\d\mu(y)\right)^{1/2}\d t\\
&\ \le\frac{1}{r_B}\int_0^\infty  \left(\int_{6B\setminus 3B}\frac{C}{\mu(B(y,\sqrt t))^2}\exp\left\{-\frac{2d(x,y)^2}{ct}\right\}
\exp\lf\{-2\epsilon\lf(1+\frac{\sqrt{t}}{\rho(x)}\r)^{{1}/{(k_0+1)}}\r\} \d\mu(y)\right)^{1/2}\d t\\
&\ \le\frac{1}{r_B}\int_0^\infty   \left(\int_{6B\setminus 3B}\frac{C}{\mu(B(x_B,\sqrt t))^2}\exp\left\{-\frac{d(y,x_B)^2}{ct}\right\}
\exp\lf\{-2\epsilon\lf(1+\frac{\sqrt{t}}{\rho(x)}\r)^{{1}/{(k_0+1)}}\r\} \d\mu(y)\right)^{1/2}\d t\\
&\ \le\frac{C}{r_B\mu(B)^{1/2}}\int_0^\infty  \exp\left\{-\frac{r_B^2}{ct}\right\}
\exp\lf\{-\epsilon\lf(1+\frac{\sqrt{t}}{\rho(x)}\r)^{{1}/{(k_0+1)}}\r\}\d t\\
&\ \le \frac{C}{r_B\mu(B)^{1/2}}\int_0^\infty \lf(\frac{t}{r^2_B}\r)^{n/4}\exp\lf\{-\epsilon\lf(1+\frac{\sqrt{t}}{\rho(x)}\r)^{{1}/{(k_0+1)}}\r\}\d t\\
&\ \le C\frac{\rho(x)^{n/2+2}}{r_B^{n/2+1}\mu(B)^{1/2}}\int_0^\infty s^{n/4}\exp\lf\{-\epsilon\lf(1+{\sqrt{s}}\r)^{{1}/{(k_0+1)}}\r\}\d s\\
&\ \le C\frac{\rho(x_B)^{n/2+1}}{r_B^{n/2}\mu(B)^{1/2}}.
\end{align*}
Similarly, there holds that
\begin{align*}
&\int_0^\infty \left(\int_{6B\setminus 3B}r_B^2|\partial_t h^v_t(x,y)|^2\d\mu(y)\right)^{1/2}\d t\\
&\ \le\int_0^\infty \frac {r_B}{t}  \left(\int_{6B\setminus 3B}\frac{C}{\mu(B(x,\sqrt t))^2}\exp\left\{-\frac{2d(x,y)^2}{ct}\right\}
\exp\lf\{-2\epsilon\lf(1+\frac{\sqrt{t}}{\rho(x)}\r)^{{1}/{(k_0+1)}}\r\} \d\mu(y)\right)^{1/2}\d t\\
&\ \le\int_0^\infty \frac {Cr_B}{t\mu(B)^{1/2}}  \exp\left\{-\frac{r_B^2}{ct}\right\}
\exp\lf\{-\epsilon\lf(1+\frac{\sqrt{t}}{\rho(x)}\r)^{{1}/{(k_0+1)}}\r\}\d t\\
&\ \le\frac {C}{r_B\mu(B)^{1/2}}\int_0^\infty   \exp\left\{-\frac{r_B^2}{ct}\right\}
\exp\lf\{-\epsilon\lf(1+\frac{\sqrt{t}}{\rho(x)}\r)^{{1}/{(k_0+1)}}\r\}\d t\\
&\ \le C\frac{\rho(x_B)^{n/2+1}}{r_B^{n/2}\mu(B)^{1/2}}.
\end{align*}
The two estimates yield that for any $x\in B(x_B,\rho(x_B))$,
\begin{eqnarray*}
|w(x)|
&&\le \frac{C}{r_B}\left(\int_{6B}{|w(y)|^2}\d\mu(y)\right)^{1/2}
\frac{\rho(x_B)^{n/2+1}}{r_B^{n/2}\mu(B)^{1/2}} \\
&&\le C\frac{\rho(x_B)^{n/2+1}}{r_B}\left(\int_{6B}\frac{|w(y)|^2}{(1+d(x_B,y))^n\mu(B(x_B,1+d(x_B,y))}\d\mu(y)\right)^{1/2} \\
&&\le C\frac{\rho(x_B)^{n/2+1}}{r_B},
\end{eqnarray*}
which tends to zero as $r_B\to\infty$.
Therefore $w\equiv 0$ on $B(x_B,\rho(x_B))$. As $x_B$ is arbitrary, we see that $w\equiv 0$.
\end{proof}

\begin{corollary}\label{Liouville-upper-space}
Let $(X,d,\mu,\mathscr{E})$ be a complete Dirichlet metric space satisfying $(D)$
 and admitting an $L^2$-Poincar\'{e} inequality.
 Suppose $0\le V\in RH_{q}(X)\cap A_\infty(X)$ with $q>(Q+1)/2$.
Suppose that $w$ is a solution to $(-\partial_t^2+\LV) w=0$ on $X \times \rr$.
If there exists $n>0$ such that
$${\int_\rr\int_X \frac{|w(y,t)|^2}{(1+t+d(x,y))^n\mu(B(x,1+t+d(x,y)))}\d\mu(y)\d t<\infty,}$$
then $w\equiv 0$.
\end{corollary}
\begin{proof}
Note that $-\partial_t^2+\LV$ is a Schr\"odinger operator on $X \times \rr$,
and by letting $V(x,t):=V(x)$, we know that $V\in RH_{q}(X\times\rr)\cap A_\infty(X\times\rr)$ with $q>(Q+1)/2$. Therefore,  Proposition \ref{Liouville} applies.
\end{proof}

 \begin{proposition}\label{Liouville-poisson}
 Let $(X,d,\mu,\mathscr{E})$ be a complete Dirichlet metric space satisfying $(D)$
 and admitting an $L^2$-Poincar\'{e} inequality.
 Suppose $0\le V\in RH_{q}(X)\cap A_\infty(X)$ with $q>(Q+1)/2$, and $u\in \HMO_\LV(X\times \rr_+)$.
For any $x\in X$ and $s,t>0$, there holds that
  \begin{align*}
  u(x,t+s)=\PV_t(u(\cdot,s))(x).
  \end{align*}
\end{proposition}
\begin{proof}
For $t> 0$, let
$$v(x,t):=u(x,t+s)-\PV_t(u(\cdot,s))(x).$$
As $u(x,t+s)$ is H\"older continuous on $X\times (-s,\infty)$ and
$u(x,s)$ is H\"older continuous on $X$, we see that
$$v(x,0):=\lim_{t\to 0^+}v(x,t)=\lim_{t\to 0^+}\left\{u(x,t+s)-\PV_t(  u(\cdot,s))(x)\right\}=0.$$
We extend $v(x,t)$ to $X\times\rr$ as
$$
w(x,t)=\begin{cases}
v(x,t), \quad & t>0;\\
0, & t=0;\\
-v(x,-t), & t<0.
\end{cases}
$$
Then $w$ is a solution to the Schr\"odinger equation
$(-\partial_t^2+ \LV) w=0$ on $X\times \rr$. By Corollary \ref{Liouville-upper-space}
and {the fact} that $w$ is odd with respect to $t$,
it is enough to show that there exists $n>0$ such that
\begin{eqnarray*}
&&\int_X\int_{0}^\infty \frac{|w(x,t)|^2}{(1+t+d(x,y))^n\mu(B(x,1+t+d(x,y)))}{\d\mu(x)}\d t<\infty.
\end{eqnarray*}
Let $n>0$ large enough to be fixed at the end of the proof.
By Proposition \ref{est-poisson-hmo}, we have
\begin{align*}
&\int_X\int_0^\infty \frac{|u(y,s+t)|^2}{(1+t+d(x,y))^n\mu(B(x,1+t+d(x,y)))}\d\mu(y)\d t\\
&\ \le \int_0^\infty   \frac{1}{(1+t)^{n-1}} \int_X \frac{|u(y,s+t)|^2}{(1+d(x,y))\mu(B(x,1+ d(x,y)))}\d\mu(y)\d t\\
&\ \le C\int_0^\infty   \frac{1}{({1+t})^{n-1}} \left( \|u(\cdot,s+t)\|^2_{L^\fz(B(x,2 ))}+{\left(1+\left|\log(s+t)\right| \right)^2}\|u\|^2_{\HMO_\LV}\right)\d t.
\end{align*}
{Apply Lemma \ref{lem3.1} to obtain}
\begin{eqnarray*}
\|u(\cdot,s+t)\|_{L^\fz(B(x,2 ))}&&\le\|u(\cdot,s+t)-u(\cdot,s)\|_{L^\fz(B(x,2 ))} +\|u(\cdot,s)\|_{L^\fz(B(x,2 ))}\\
&&\le \left\|\int_{s}^{s+t}{|\partial_ru(\cdot,r)|}\d r\r\|_{L^\fz(B(x,2 ))} +\|u(\cdot,s)\|_{L^\fz(B(x,2 ))}\\
&&\le C\|u\| _{\HMO_\LV}\log\lf(\frac{s+t}{s}\r)+\|u(\cdot,s)\|_{L^\fz(B(x,2 ))}.
\end{eqnarray*}
{Therefore one has}
\begin{align}\label{est-liouville-1}
&\int_X\int_0^\infty \frac{|u(y,s+t)|^2}{(1+d(x,y)+t)^n\mu(B(x,1+t+d(x,y)))}\d\mu(y)\d t\nonumber\\
&\ \le \int_0^\infty   \frac{C\left(\left(\log\frac{s+t}{s}\right)^2\|u\|^2_{\HMO_\LV}+\|u(\cdot,s)\|^2_{L^\fz(B(x,2 ))}+\left(1+\left|\log(s+t)\right| \right)^2\|u\|^2_{\HMO_\LV}\right)}{(1+t)^{n-1}}  \d t\nonumber\\
&\ \le C(s,\|u\|_{\HMO_\LV},\|u(\cdot,s)\|_{L^\fz(B(x,2 ))})<\infty.
\end{align}

For the remaining term, note that by the estimate in Corollary \ref{lem3.2}, it holds
for all $t>0$,
\begin{eqnarray*}
|\PV_t(u(\cdot,s))(x)|&&\le C|\PV_t(1)(x)|^{1/2}\left(\int_X\frac{t|u(y,s)|^2}{(t+d(x,y))\mu(B(x,t+d(x,y)))}\d\mu(y)\right)^{1/2}\nonumber\\
&&\le C\left(\|u(\cdot,s)\|_{L^\fz(B(x,2t))}+\left(1+\left|\log\frac{t}{s}\right| \right)\|u\|_{\HMO_\LV}\right),
\end{eqnarray*}
which implies for $t\le 1$,
\begin{eqnarray*}
|\PV_t(  u(\cdot,s))(x)|&&\le
C\left( \|u(\cdot,s)\|_{L^\fz(B(x,2))}+C\left(1+\left|\log\frac{t}{s}\right| \right)\|u\|_{\HMO_\LV}\right),
\end{eqnarray*}
and for $t>1$,
\begin{eqnarray*}
|\PV_t(  u(\cdot,s))(x)|&&\le |\PV_t( 1)(x)|^{1/2}\left( \int_X\frac{t|u(y,s)|^2}{(t+d(x,y))\mu(B(x,t+d(x,y)))}\d\mu\right)^{1/2}\\
&&\le C\left( \int_X\frac{t|u(y,s)|^2}{(1+d(x,y))\mu(B(x,1+d(x,y)))}\d\mu\right)^{1/2}\\
&&\le C\sqrt t \left( \|u(\cdot,s)\|_{L^\fz(B(x,2 ))}+C\left(1+\left|\log s\right| \right)\|u\| _{\HMO_\LV}\right).
\end{eqnarray*}
Therefore, it holds that
\begin{align*}
&\int_X\int_0^\infty \frac{|\PV_t(u(\cdot,s))(y)|^2}{(1+t+d(x,y))^n\mu(B(x,1+t+d(x,y)))}\d\mu(y)\d t\\
&\ \le \int_0^\infty \int_X \frac{C(1+t) \left( \|u(\cdot,s)\|^2_{L^\fz(B(x,2 ))}+\left(1+{\left|\log s\right|}
+\left|\log t\right| \right)^2\|u\|^2_{\HMO_\LV}\right)}{(1+t)^{n-1}(1+t+d(x,y)) \mu(B(x,1+t+d(x,y)))} \d\mu(y)\d t\\
&\ \le \int_0^\infty \frac{C}{(1+{t})^{n-2}}\left( \|u(\cdot,s)\|^2_{L^\fz(B(x,2 ))}+C\left(1+\left|\log s\right| +\left|\log t\right| \right)^2\|u\|^2_{\HMO_\LV}\right)  \d t\\
&\ \le C(s,\|u\|_{\HMO_\LV},\|u(\cdot,s)\|_{L^\fz(B(x,2 ))})<\infty,
\end{align*}
which together with \eqref{est-liouville-1} yields that
\begin{eqnarray*}
&&\int_X\int_{-\infty}^\infty \frac{|w(y,t)|^2}{(1+t+d(x,y))^n\mu(B(x,1+t+d(x,y)))}\d\mu(y)\d t<\infty,
\end{eqnarray*}
{provided $n>3$.}
Corollary \ref{Liouville-upper-space} then implies
$w(x,t)\equiv 0,$
which means $u(x,t+s)\equiv\PV_t(u(\cdot,s))(x)$, and thus completes the proof.
\end{proof}

We have following uniform norm estimate.
\begin{lemma}\label{uniform-hmol}
Let $(X,d,\mu,\mathscr{E})$ be a complete Dirichlet metric space satisfying $(D)$
and admitting an $L^2$-Poincar\'{e} inequality.
Suppose $0\le V\in RH_{q}(X)\cap A_\infty(X)$ with $q>(Q+1)/2$, and  $u\in \HMO_\LV(X\times \rr_+)$.
Then for any $s>0$, it holds
\begin{eqnarray*}
\sup_{B}\lf(\int_0^{r_B}\fint_{B}|t\sqrt{\LV}\PV_t(u(\cdot,s))|^2\d \mu\frac{\d t}{t}\r)^{1/2}&&\le
C\|u\|_{\HMO_\LV}.
\end{eqnarray*}
\end{lemma}
\begin{proof}
By Proposition \ref{Liouville-poisson}, there holds that
\begin{eqnarray*}
\lf(\int_0^{r_B}\fint_{B}|t\sqrt{\LV}\PV_tu_s|^2\d \mu\frac{\d t}{t}\r)^{1/2}
&&= \lf(\int_{0}^{r_B}{\fint_{B}}|t\partial_t u(y,t+s)|^2\d \mu(y)\frac{\d t}{t}\r)^{1/2},
\end{eqnarray*}
{where and in what follows, we set $u_s(\cdot)=u(\cdot,s)$.}
If $r_B>s$, then
\begin{eqnarray*}
\lf(\int_0^{r_B}\fint_{B}|t\sqrt{\LV}\PV_tu_s|^2\d \mu\frac{\d t}{t}\r)^{1/2}&&\le  \lf(\frac{\mu(B(x_B,r_B+s))}{\mu(B(x_B,r_B))}\int_{0}^{r_B+s}
\fint_{B(x_B,r_B+s)}|t\partial_t u(y,t)|^2\d \mu(y)\frac{\d t}{t}\r)^{1/2}\nonumber\\
&&\le C\|u\|_{\HMO_\LV}.
\end{eqnarray*}
Otherwise $r_B\le s$, Lemma \ref{lem3.1} implies that
\begin{eqnarray*}
\lf(\int_0^{r_B}\fint_{B}|t\sqrt{\LV}\PV_t u_s|^2\d \mu\frac{\d t}{t}\r)^{1/2}&&\le  \lf(\int_{0}^{r_B}{\fint_{B}}\frac{Ct^2}{(t+s)^2}\|u\|^2_{\HMO_\LV}\d \mu\frac{\d t}{t}\r)^{1/2}\nonumber\\
&&\le C\|u\|_{\HMO_\LV},
\end{eqnarray*}
as required.
\end{proof}
We next establish an identity between $u\in \HMO_\LV(X\times\rr_+)$ and  $f\in H^1_\rho(X)$
which can be written as the sum of finite atoms.
The proof is similar to that of  \cite[Proposition 5.1]{DY2005-1}.

\begin{lemma}\label{reproducing-formula}
 Let $(X,d,\mu,\mathscr{E})$ be a complete Dirichlet metric space satisfying $(D)$
and admitting an $L^2$-Poincar\'{e} inequality.
Suppose $0\le V\in RH_{q}(X)\cap A_\infty(X)$ with $q>(Q+1)/2$, and $u\in \HMO_\LV(X\times \rr_+)$.
For $f\in H^1_\rho(X)$ which can be written as the sum of finite atoms, it holds
\begin{align*}
\int_X fu_\epsilon \d\mu=4\int_0^\fz\int_Xt\sqrt{\LV}\PV_t u_\epsilon t\sqrt{\LV}\PV_t f\d \mu\frac{\d t}{t}.
\end{align*}
\end{lemma}
\begin{proof}
For $f\in H^1_\rho(X)$
which can be written as the sum of finite atoms, $f$ has compact support and
belongs to $L^2(X)$.  Suppose that $\supp f\subset B:=B(x_B,r_B)$.

{\bf Step 1.}
By Theorem \ref{hardy-equiv}, it holds
\begin{align*}
\left\|\mathcal{S}_\PV(f)\right\|_{L^1}&=\int_X\left(\int_0^\fz\fint_{B(x,{t})}
|t\sqrt{\LV}\PV_t(f)(z)|^2{\d \mu(z)}\frac{\d t}{t}\right)^{1/2}\d\mu(x)\le C\|f\|_{H^1_\rho}.
\end{align*}
By \cite[Proposition 4.10]{HLMMY2011} (see \cite{JY2011,Ru07}), we can write
$$t\sqrt{\LV}\PV_t(f)(x)=\sum_{j}\lz_ja_j(x,t),$$
where $\sum_j|\lz_j|\sim \left\|\mathcal{S}_\PV(f)\right\|_{L^1}$, and
for each $j$, there exists a ball $B_j\subset X$ such that,
 $\supp a_j\subset B_j\times (0,r_{B_j})$, and
 $$\int_0^{r_{B_j}}{\int_{B_j}}|a_j(x,t)|^2\frac{\d\mu(x)\d t}{t}\le{ \frac{1}{\mu(B_j)}}.$$

On the other hand, Lemma \ref{uniform-hmol} gives
\begin{eqnarray*}
\sup_{B}\lf(\int_0^{r_B}\fint_{B}|t\sqrt{\LV}\PV_t({u(\cdot,\epsilon)})|^2\d \mu\frac{\d t}{t}\r)^{1/2}&&\le
C\|u\|_{\HMO_\LV},
\end{eqnarray*}
and hence,
\begin{align}\label{hardy-bmo}
&\int_0^\infty\int_X{\left|t\sqrt{\LV}\PV_t(u(\cdot,\epsilon))(x)\right|}\left|t\sqrt{\LV}\PV_t f(x)\right|\d\mu(x)\frac{\d t}{t} \nonumber\\
&\ \le \sum_j|\lz_j|\int_0^{r_{B_j}}\int_{{B_j}}\left|t\sqrt{\LV}
\PV_t({u(\cdot,\epsilon)})(x)\right|\left|a_j(x,t)\right|\d\mu(x)\frac{\d t}{t}\nonumber\\
&\ \le C \sum_j|\lz_j| \left(\int_0^{r_{B_j}}\int_{B_j}|a_j(x,t)|^2\frac{\d\mu(x)\d t}{t}\right)^{1/2}
 \left(\int_0^{r_{B_j}}\int_{B_j}|t\sqrt{\LV}\PV_t({u(\cdot,\epsilon)})(x)|^2\frac{\d\mu(x)\d t}{t}\right)^{1/2}\nonumber \\
&\ \le C\|f\|_{H^1_\rho} \|u\|_{\HMO_\LV}.
\end{align}

\textbf{Step 2.} Let us verify the required identity,
\begin{align*}
\int_X u_\epsilon f\d\mu=4\int_0^\fz\int_Xt\sqrt{\LV}\PV_t u_\epsilon t\sqrt{\LV}\PV_t f\d \mu\frac{\d t}{t}.
\end{align*}
By \eqref{hardy-bmo}, we see that
\begin{align*}
\int_0^\fz\int_X t\sqrt{\LV}\PV_t u_\epsilon t\sqrt{\LV}\PV_t f\d \mu\frac{\d t}{t}&=
\lim_{\substack{\gamma\to 0^+,\, R\to+\fz}}\int_{\gamma}^R\int_X t\sqrt{\LV}\PV_t u_\epsilon t\sqrt{\LV}\PV_t f\d \mu\frac{\d t}{t}\\
& =\lim_{\substack{\gamma\to 0^+,\, R\to+\fz}}\int_{\gamma}^R \int_X u_\epsilon \,t^2{\LV}\PV_{2t} f\d \mu\frac{\d t}{t}\\
& =\lim_{\substack{\gamma\to 0^+,\, R\to+\fz}}\int_X u_\epsilon\lf(\int_{\gamma}^Rt^2{\LV}\PV_{2t} f\frac{\d t}{t}\r)\d\mu \\
&=\lim_{\substack{\gamma\to 0^+,\, R\to+\fz}}\int_{4B}\cdots\d \mu+\lim_{\substack{\gamma\to 0^+ \\ R\to+\fz}}\int_{(4B)^\complement}\cdots\d \mu.
\end{align*}

For the local part, as $u$ is locally H\"older continuous on $X\times \rr_+$,
it holds $u(\cdot,\epsilon)\in L^2(4B)$. Moreover,
by the spectral theory, one has
$$f(x)=4\int_0^\fz t^2{\LV}\PV_{2t} f(x)\d \mu\frac{\d t}{t}=\lim_{\substack{\gamma\to 0^+,\, R\to+\fz}}4\int_\gamma^R t^2{\LV}\PV_{2t} f(x)\d \mu\frac{\d t}{t}$$
in $L^2(X)$.
Therefore, it holds
\begin{align}\label{converge-1}
\lim_{\substack{\gamma\to 0^+,\, R\to+\fz}}\int_{4B}u_\epsilon\lf(\int_{\gamma}^R t^2{\LV}\PV_{2t} f(x)\frac{\d t}{t}\r)\d \mu
=\frac14\int_{4B}u_\epsilon f\d \mu.
\end{align}

To compute the global part, we first estimate the inner integral.
Recall that $\supp f\subset B=B(x_B,r_B)$.
It follows from Proposition \ref{est-poisson-kernel} and Lemma \ref{crit-func-pro-1} that for any $t>0$
and $x\in X\setminus 4B$,
\begin{eqnarray*}
|t^2{\LV}\PV_{2t} f(x)|&& \le C\int_{B}\frac{t|f(y)|}{(t+d(x,y)\mu(B(x,t+d(x,y)))}\lf(1+\frac{t+d(x,y)}{\rho(y)}\r)^{-N}\d\mu(y)\\
&&\le C\|f\|_{L^1(B)} \frac{t}{(t+d(x,x_B)\mu(B(x,t+d(x,x_B)))}\lf(\frac{\rho(x_B)\lf(1+\frac{r_B}{\rho(x_B)}\r)^{{k_0}/{(k_0+1)}}}{t+d(x,x_B)}\r)^{N}\\
&&\le C(x_B,\rho(x_B),r_B)\|f\|_{L^1(B)} \frac{t}{(t+d(x,x_B))^{N+1}\mu(B(x_B,t+d(x,x_B)))},
\end{eqnarray*}
where $N>0$ large enough. Hence there holds
\begin{eqnarray*}
\int_0^\infty|t^2{\LV}\PV_{2t} f(x)|\frac{\d t}{t}&&
\le \int_0^\infty\frac{C(x_B,\rho(x_B),r_B)\|f\|_{L^1(B)}}{(t+d(x,x_B))^{N+1}\mu(B(x_B,t+d(x,x_B)))}\d t\\
&&\le \frac{C(x_B,\rho(x_B),r_B)\|f\|_{L^1(B)}}{{d(x,x_B)}\mu(B(x_B,d(x,x_B)))}.
\end{eqnarray*}
By this estimate and  Proposition \ref{est-poisson-hmo}, we then obtain
\begin{align*}
&\int_{X\setminus 4B}\int_{\gamma}^R\left|u_\epsilon(x)  t^2{\LV}\PV_{2t} f(x)\r|\frac{\d t}{t}\d \mu(x)\\
&\ \le \int_{X\setminus 4B}\frac{C(x_B,\rho(x_B),r_B)\|f\|_{L^1(B)} |u(x,\epsilon)|}{{d(x,x_B)}\mu(B(x_B,d(x,x_B)))}\d\mu(x)\\
&\ \le  C(x_B,\rho(x_B),r_B)\|f\|_{L^1(B)} \int_{X\setminus 4B}\frac{|u(x,\epsilon)|}{{(r_B+d(x,x_B))}\mu(B(x_B,r_B+d(x,x_B)))}\d\mu(x)\\
&\  \le C(x_B,\rho(x_B),r_B)\|f\|_{L^1(B)} {\lf(\int_X\frac{|u(x,\epsilon)|^2}{{(r_B+d(x,x_B))}\mu(B(x_B,r_B+d(x,x_B)))}\d\mu(x)\r)^{1/2}}\\
&\  \le  C(x_B,\rho(x_B),r_B)\|f\|_{L^1(B)}\left(\|u(\cdot,\epsilon)\|_{L^\fz(B(x,2r_B))}
+\left(1+\left|\log\frac{r_B}{\epsilon}\right| \right){\|u\|_{\HMO_\LV}}\right)\\
&\   <\infty.
\end{align*}
This together with \eqref{converge-1} implies that
\begin{align*}
\int_0^\fz\int_X t\sqrt{\LV}\PV_t u_\epsilon t\sqrt{\LV}\PV_t f\d \mu\frac{\d t}{t}=\lim_{\substack{\gamma\to 0^+,\, R\to+\fz}}\int_{X}u_\epsilon\lf(\int_{\gamma}^R t^2{\LV}\PV_{2t} f(x)\frac{\d t}{t}\r)\d \mu
=\frac14\int_X u_\epsilon f\d \mu,
\end{align*}
and completes the proof.
\end{proof}

\begin{lemma}\label{existence-bmo}
Let $(X,d,\mu,\mathscr{E})$ be a complete Dirichlet metric space satisfying $(D)$
and admitting an $L^2$-Poincar\'{e} inequality.
 Suppose $0\le V\in RH_{q}(X)\cap A_\infty(X)$ with $q>(Q+1)/2$, and $u\in \HMO_\LV(X\times \rr_+)$.
Then for any $s>0$, $u(\cdot,s)\in \BMO_\LV(X)$. Moreover, there exits
$C>0$ such that for any $s>0$ it holds
\begin{align*}
\|u(\cdot,s)\|_{\BMO_\LV}\le C\|u\|_{\HMO_\LV}.
\end{align*}
\end{lemma}
\begin{proof}
For any $f\in H^1_\rho(X)$ which can be written as the sum of finite atoms, it follows from Lemma \ref{reproducing-formula} that
\begin{align*}
\int_X f(x)u(x,s)\d\mu(x)=4\int_0^\fz\int_Xt\sqrt{\LV}\PV_t u_s t\sqrt{\LV}\PV_t f\d \mu\frac{\d t}{t}.
\end{align*}
By \eqref{hardy-bmo}, we have
\begin{align*}
\left|\int_X f(x)u(x,s)\d\mu(x)\right|\le C\|f\|_{H^1_\rho}\|u\|_{\HMO_\LV},
\end{align*}
where $C$ is independent of $s$. It is obvious that the class of
sum of finite atoms is dense in $H^1_\rho(X)$.
We then conclude from the duality $(H^1_\rho(X))^\ast=\BMO_\LV(X)$ (see Theorem \ref{hardy-equiv}) that $u(\cdot,s)\in \BMO_\LV(X)$ with $\|u(\cdot,s)\|_{\BMO_\LV}\le C\|u\|_{\HMO_\LV}.$
\end{proof}

\begin{proposition}\label{HmotoBmo}
Let $(X,d,\mu,\mathscr{E})$ be a complete Dirichlet metric space satisfying $(D)$
and admitting an $L^2$-Poincar\'{e} inequality.
 Suppose $0\le V\in RH_{q}(X)\cap A_\infty(X)$ with $q>(1+Q)/2$.
Then if $u\in \HMO_{\LV}(X\times \rr_+)$, there exists a function $f\in \BMO_\LV(X)$ such that $u(x,t)=\PV_tf(x)$.
Moreover there exists a constant $C>1$ such that
$$\|f\|_{\BMO_\LV}\le C\|u\|_{\HMO_{\LV}} .$$
  \end{proposition}
  \begin{proof}
We deduce that from Lemma \ref{existence-bmo}  that the family $\{u_\epsilon(\cdot)\}_{0<\epsilon<1}$ is uniformly bounded in $\BMO_\LV(X)$,
which, together with the Banach-Alaoglu theorem, yields that
there exist sequence $\epsilon_k\to0$ ($k\to\fz$) and function $f\in \BMO_\LV(X) $ such that  $u_{\epsilon_k}\to f$ { weak-$*$} in $\BMO_\LV(X)$,
and
$$\|f\|_{\BMO_\LV}\le C\|u\|_{\HMO_\LV}.$$

On the other hand, by Theorem \ref{hardy-equiv}
and Proposition \ref{poisson-hardy}, one has
$(H^1_\LV({X}))^\ast=\BMO_\LV(X)$ and {${\partial_t p_t^v(\cdot\,,y)}\in H^1_\LV(X)=H^1_\rho(X)$},
and therefore
\begin{align*}
{\int_X \partial_tp_t^v(x,y)u_{\epsilon_k}(y)\d \mu(y)\to \int_X {\partial_tp^v_t(x,y)}f(y)\d \mu(y),\quad k\to\fz.}
\end{align*}
This together with Proposition \ref{Liouville-poisson} implies that $\partial_tu(x,t)=\partial_t\PV_tf(x)$,
and therefore there exists a function $h(x)$ such that $u(x,t)=\PV_tf(x)+h(x)$. Since $u(x,t)$ and $\PV_tf(x)$ are $\LV_+$-harmonic functions on $X\times\rr_+$, they are locally H\"older continuous. Hence $h$ is also $\LV_+$-harmonic function on $X\times\rr_+$. Since $h$ is independent of $t$, $h$ satisfies
$\LV u=0$ on $X$.

Let us show that $h\equiv 0$. We fix a $t>0$.
For each $y\in X$, by letting $B=B(y,t)$ and using Proposition \ref{est-poisson-kernel}, one has
\begin{eqnarray}\label{poisson-bmo}
|\PV_tf(y)|&&\le C\int_X\frac{t|f(x)-f_B|}{(t+d(x,y))\mu(B(y,t+d(x,y)))}{\d\mu(x)}+|f_{B}|\nonumber\\
&&\le C\int_{B}\frac{t|f(x)-f_B|}{t\mu(B(y,t))}\d\mu(x)+
C\sum_{j=1}^\infty\int_{2^jB\setminus 2^{j-1}B}\frac{t|f(x)-f_B|}{2^jt\mu(B(y,2^jt))}\d\mu(x)
+ |f_{B}|\nonumber\\
&&\le C\|f\|_{\BMO_\LV}+C\sum_{j=1}^\infty2^{-j} \left(\fint_{2^jB}\left(|f(x)-f_{2^jB}|+|f_B-f_{2^jB}|\right)\d\mu(x)\right)+|f_B|\nonumber\\
&&\le C\|f\|_{\BMO_\LV}+C\sum_{j=1}^\infty2^{-j}j\|f\|_{\BMO_\LV}+|f_B|\nonumber\\
&&\le C\|f\|_{\BMO_\LV}+|f_B|.
\end{eqnarray}
If $t\ge\rho(y)$, then $|f_B|\le \|f\|_{\BMO_\LV}$. Otherwise $t<\rho(y)$, by Lemma \ref{crit-func-pro-1}, it holds
for a fixed $x_B\in X$
\begin{eqnarray*}
|f_B|&&\le |f_{B}-f_{B(y,\rho(y))}|+|f_{B(y,\rho(y))}|\le C\left(1+\log\frac{\rho(y)}{t}\right)\|f\|_{\BMO_\LV}\\
&&\le C\left(1+\log \frac{\rho(x_B)}{t}+\frac{k_0}{k_0+1}\log\lf(1+\frac{d(x_B,y)}{\rho(x_B)}\r)\right)\|f\|_{\BMO_\LV}\\
&&\le C\left(1+\log \frac{\rho(x_B)}{t}+\frac{d(x_B,y)}{\rho(x_B)}\r)\|f\|_{\BMO_\LV}.
\end{eqnarray*}
The above two estimates imply
\begin{eqnarray*}
\int_X\frac{|\PV_tf(y)|^2}{(1+d(x,y))^3\mu(B(x,1+d(x,y)))}\d\mu(y)
&&\le C\int_X\frac{\left(1+\left|\log \frac{\rho(x_B)}{t}\right|^2+\frac{{ d(x_B,y)^2}}{\rho(x_B)^2}\r)\|f\|^2_{\BMO_\LV}}
{{ (1+d(x,y))^3}\mu(B(x,1+d(x,y)))}\d\mu(y)\\
&&\le C(x_B,\rho(x_B),t)\|f\|^2_{\BMO_\LV}<\infty.
\end{eqnarray*}
On the other hand,
Proposition \ref{est-poisson-hmo} shows
\begin{eqnarray*}
&&\int_X\frac{|u(y,t)|^2}{(1+d(x,y))^3\mu(B(x,1+d(x,y)))}\d\mu(y)
\le C\left(\|u(\cdot,t)\|^2_{L^\fz(B(x,2))}+ \left(1+\left|\log t\right| \right)^2\|u\|^2_{\HMO_\LV}\right).
\end{eqnarray*}
We therefore can conclude via the equality $h(x)=u(x,t)-\PV_tf(x)$ that
\begin{align*}
&\int_X\frac{|h(y)|^2}{(1+d(x,y))^3\mu(B(x,1+d(x,y)))}\d\mu(y) \\
&\ \le 2\int_X\frac{|u(y,t)|^2}{(1+d(x,y))^3\mu(B(x,1+d(x,y)))}\d\mu(y)+2\int_X\frac{|\PV_tf(y)|^2}{(1+d(x,y))^3\mu(B(x,1+d(x,y)))}\d\mu(y)\\
&\ <\infty.
\end{align*}
Finally Proposition \ref{Liouville} shows that $h\equiv 0$, and therefore
$u(x,t)=\PV_tf(x)$ with $$\|f\|_{\BMO_\LV}\le C\|u\|_{\HMO_\LV},$$
which completes the proof.
\end{proof}

\section{From BMO to HMO}
\hskip\parindent In this section, we complete the proof of the main result by showing that every $\BMO_\LV(X)$ function $f$ induces a Carleson measure $|{t}\nabla\PV_tf|^2\d \mu\frac{\d t}{t}$.

The proof of that, the time derivative part $|{t}\partial_t \PV_tf|^2\d \mu\frac{\d t}{t}$
is a Carleson measure, is similar to \cite{DY2005-1,DGMTZ2005,LinLiu2011}.
We provide a proof for completeness.

\begin{lemma}\label{est-time-carleson}
Let $(X,d,\mu,\mathscr{E})$ be a complete Dirichlet metric space satisfying $(D)$
and admitting an $L^2$-Poincar\'{e} inequality.
Suppose $0\le V\in RH_{q}(X)\cap A_\infty(X)$ with $q>\max\{1,Q/2\}$, and  $f\in \BMO_\LV(X)$.
For any ball $B=B(x_B,r_B)$, it holds
\begin{align*}
\int_0^{r_B}\int_B|{t}\partial_t \PV_tf|^2\d \mu\frac{\d t}{t}\le C\mu(B)\,\|f\|^2_{\BMO_\LV}.
\end{align*}
\end{lemma}
\begin{proof}
Note first that
\begin{eqnarray}\label{est-bmo-carleson-1}
\int_0^{r_B}\int_B|{t}\partial_t \PV_t((f-f_{4B})\mathbbm{1}_{4B})|^2\d \mu\frac{\d t}{t}&&\le
\int_0^{\infty}\int_X|{t}\partial_t \PV_t((f-f_{4B})\mathbbm{1}_{4B})|^2\d \mu\frac{\d t}{t}\nonumber\\
&&\le C\|f-f_{4B}\|^2_{L^2(4B)}\nonumber\\
&&\le C\mu(B)\|f\|^2_{\BMO_\LV}.
\end{eqnarray}
Then, similar to the estimate \eqref{poisson-bmo}, by using Proposition \ref{est-poisson-kernel},
we deduce that, for any $x\in B$,
\begin{eqnarray}\label{est-bmo-carleson-2v}
|{t}\partial_t \PV_t((f-f_{4B})\mathbbm{1}_{X\setminus 4B})(x)|
&&\le \int_{X\setminus 4B}\frac{t|f(y)-f_{4B}|}{(t+d(x,y))\mu(B(y,t+d(x,y)))}\d\mu(y)\le \frac{Ct}{r_B}{\|f\|_{\BMO_\LV}},
\end{eqnarray}
and therefore,
\begin{eqnarray}\label{est-bmo-carleson-2}
\int_0^{r_B}\int_B|{t}\partial_t \PV_t((f-f_{4B})\mathbbm{1}_{X\setminus 4B})|^2\d \mu\frac{\d t}{t}
&&\le \int_0^{r_B}\int_B \frac{t^2}{r_B^2}{\|f\|^2_{\BMO_\LV}} \d \mu\frac{\d t}{t} \\ \nonumber
&&\le C\mu(B)\|f\|^2_{\BMO_\LV}.
\end{eqnarray}
Finally, by Proposition \ref{est-poisson-kernel},  for any $N>0$ it holds
\begin{align*}
   |t \partial_t\e^{-t\sqrt{\LV}}(f_{4B})(x)|\le C \lf(\dfrac{{t}}{\rho(x)}\r)^{\dz}\lf(1+\frac{{t}}{\rho(x)}\r)^{-N}|f_{4B}|,
\end{align*}
where $\delta\in(0,\min\{1,2-Q/q\})$.
If $4r_B<\rho(x_B)$, then
\begin{eqnarray*}
|f_{4B}|&&\le |f_{4B}-f_{B(x_B,\rho(y_B))}|+|f_{B(x_B,\rho(x_B))}|\le C\left(1+\log\frac{\rho(x_B)}{r_B}\right)\|f\|_{\BMO_\LV},
\end{eqnarray*}
which, together with Lemma \ref{crit-func-pro-1}, implies
\begin{eqnarray}\label{est-bmo-carleson-3}
\int_0^{r_B}\int_B|{t}\partial_t \PV_t(f_{4B})(x)|^2\d \mu(x)\frac{\d t}{t}
&&\le C\int_0^{r_B}\int_B \lf(\dfrac{{t}}{\rho(x)}\r)^{2\dz}\left(1+\log\frac{\rho(x_B)}{r_B}\right)^2\|f\|^2_{\BMO_\LV} \d \mu(x)\frac{\d t}{t}\nonumber\\
&&\le C\mu(B)\lf(\dfrac{{r_B}}{\rho(x_B)}\r)^{2\dz}\left(1+\log\frac{\rho(x_B)}{r_B}\right)^2\|f\|^2_{\BMO_\LV}\nonumber\\
&&\le C\mu(B)\|f\|^2_{\BMO_\LV}.
\end{eqnarray}
If $4r_B\ge \rho(x_B)$,  then $|f_{4B}|\le \|f\|_{\BMO_\LV}$. By Fubini's theorem, we  conclude that
\begin{eqnarray}\label{est-bmo-carleson-4}
\int_{0}^{r_B}\int_{B}|{t}\partial_t \PV_t(f_{4B})(x)|^2\d\mu(x)\frac{\d t}{t}
 &&\le C\int_0^{r_B}\int_{B} \left(\frac{t}{\rho(x)}\right)^{2\delta} \left(1+\frac{t}{\rho(x)}\right)^{-2N}|f_{4B}|^2\frac{\d \mu(x)\d t}{t}\nonumber\\
 &&\le C|f_{4B}|^2\int_{B} \int_0^{\fz}\left(\frac{t}{\rho(x)}\right)^{2\delta} \left(1+\frac{t}{\rho(x)}\right)^{-2N}\frac{\d t}{t}\d \mu(x)\nonumber\\
 &&\le C|f_{4B}|^2\int_{B} \int_0^{\fz}\frac{s^{2\delta}}{\left(1+s\right)^{2N}}\frac{\d s}{s}\d \mu(x)\nonumber\\
 &&\le C\mu(B)\|f\|_{\BMO_\LV}^2,
\end{eqnarray}
as soon as we choose $N>\delta$.
A combination of the estimates \eqref{est-bmo-carleson-1},  \eqref{est-bmo-carleson-2},  \eqref{est-bmo-carleson-3} and \eqref{est-bmo-carleson-4} completes the proof.
\end{proof}

We next estimate the space derivation part $|{t}\nabla_x \PV_t(f)|^2\d \mu\frac{\d t}{t}$,
which seems need more work due to lack of pointwise bound for the space derivative of the Poisson kernel.
\begin{proposition}\label{key-cacc}
Let $(X,d,\mu,\mathscr{E})$ be a complete Dirichlet metric space satisfying $(D)$
and admitting an $L^2$-Poincar\'{e} inequality.
Assume that $0\le V\in RH_{q}(X)\cap A_\infty(X)$ with $q>\max\{1,Q/2\}$.
Suppose that $g$ satisfies for some $y\in X$ that
$$\int_{X}\frac{|g(x)|}{(1+d(x,y))\mu(B(y,1+d(x,y)))}\d\mu(x)<\infty.$$
Then for any ball $B=B(x_B,r_B)$, it holds
\begin{eqnarray}\label{key-est}
\int_{0}^{r_B}\int_{B}|{t}\nabla_x \PV_tg|^2\d\mu\frac{\d t}{t}
 \le C\int_0^{2r_B}\int_{2B} \left(|t^2\partial^2_t\PV_tg||\PV_tg|+\frac{t^2}{r_B^2} |\PV_tg|^2 \right)\frac{\d \mu\d t}{t}.
\end{eqnarray}
Moreover, for any constant $c_0\neq 0$, it holds
\begin{align}\label{key-est-2}
\int_{0}^{r_B}\int_{B}|{t}\nabla_x \PV_tg|^2\d\mu\frac{\d t}{t}
&\le C\int_0^{2r_B}\int_{2B} \left(|t^2\partial^2_t \PV_tg||\PV_tg-c_0|+\frac{t^2}{r_B^2}  |\PV_tg-c_0|^2 \right)\frac{\d \mu\d t}{t}\nonumber\\
&\ +C\int_0^{2r_B}\int_{2B} t^2|\PV_t g||\PV_tg-c_0| V \frac{\d \mu\d t}{t}.
\end{align}
\end{proposition}
\begin{proof}
We take a Lipschitz function $\vz$ on $X$ with $\supp\vz\subset 2B$ such that $\vz=1$ on $B$ and  $|\nabla_x\vz|\le C/{r_B}$, and for each $\epsilon\in (0,r_B)$, take a $C^\infty(\rr)$ function $\phi_\epsilon(t)$
such that {$\supp \phi_\epsilon\subset (\epsilon,2r_B)$}, $\phi_\epsilon(t)=1$ on $(2\epsilon,r_B)$,
$|\partial_t\phi_\epsilon(t)|\le C/\epsilon$ for $t\in (\epsilon,2\epsilon)$, $|\partial_t\phi_\epsilon(t)|\le C/r_B$ for $t>r_B$.

Since $g$ satisfies
$$\int_{X}\frac{|g(x)|}{(1+d(x,y))\mu(B(y,1+d(x,y)))}\d\mu(x)<\infty,$$
we know from Proposition \ref{est-poisson-kernel} that
$\PV_t g$, $\partial_t\PV_t g$, $\partial^2_t\PV_t g$, are locally bounded.
Moreover, for any compactly supported Lipschitz function $\psi(x,t)$ on $X\times \rr_+$, it holds
\begin{align*}
\int_0^\fz\int_{X}|\nabla_x \PV_t g|^2\psi^ 2\d\mu\d t
&=\int_0^\fz\int_{X}\left[\langle\nabla_x \PV_t g, \nabla_x(\PV_t g \psi^ 2)\rangle
-2 \langle\nabla_x \PV_t g, \nabla_x\psi\rangle\psi\PV_t g\right] \d\mu\d t\\
&\le \int_0^\fz\int_{X}\langle\LV\PV_t g, \PV_t g \psi^ 2\rangle\d\mu\d t-\int_0^\fz\int_{X}\langle V\PV_t g, \PV_t g \psi^ 2\rangle\d\mu\d t\\
&\ +\frac12 \int_0^\fz\int_{X}|\nabla_x \PV_t g|^2\psi^ 2\d\mu\d t+ 2\int_0^\fz\int_{X}|\nabla_x \psi|^2 |\PV_t g|^2\d\mu\d t,
\end{align*}
and hence,
\begin{eqnarray*}
\int_0^\fz\int_{X}|\nabla_x \PV_t g|^2\psi^ 2\d\mu\d t
&&\le 2\int_0^\fz\int_{X}\langle \partial^2_t\PV_t g, \PV_t g \psi^ 2\rangle\d\mu \d t+4\int_0^\fz\int_{X}|\nabla_x \psi|^2 |\PV_t g|^2\d\mu \d t.
\end{eqnarray*}
This implies that $\PV_t g\in W^{1,2}_\loc(X\times\rr_+)$.

Let $c_0$ be a constant. As $t(\vz\phi_\epsilon)^2\PV_tg\in W^{1,2}(X\times\rr_+)$  with compact support in $2B\times (\epsilon,2r_B)$, we have
\begin{align}\label{cacc-par-1}
\int_0^{2r_B}\int_{2B}|{t}\nabla_x \PV_tg|^2(\vz\phi_\epsilon)^2\d \mu\frac{\d t}{t}
&=\int_0^{2r_B}\int_{2B}|{t}\nabla_x (\PV_tg-c_0)|^2(\vz\phi_\epsilon)^2\d \mu\frac{\d t}{t}\nonumber\\
&=\int_0^{2r_B}\int_{2B} \langle \nabla_x \PV_tg, \nabla_x(t(\vz\phi_\epsilon)^2 (\PV_tg- c_0))\rangle \d \mu\d t\nonumber\\
&\  -\int_0^{2r_B}\int_{2B} \langle \nabla_x \PV_tg, \nabla_x \vz\rangle 2t\vz\phi_\epsilon^2 (\PV_tg- c_0)\d \mu\d t.
\end{align}
Since $\LV=\L+V$, we have
\begin{align}\label{cacc-par-2}
&\int_0^{2r_B}\int_{2B} \langle \nabla_x \PV_tg, \nabla_x(t(\vz\phi_\epsilon)^2 (\PV_tg- c_0))\rangle {\d\mu}\d t\nonumber\\
&\ =\int_0^{2r_B}\int_{2B} \langle \L \PV_tg, t(\vz\phi_\epsilon)^2 (\PV_tg-c_0)\rangle \d \mu\d t\nonumber\\
&\ =\int_0^{2r_B}\int_{2B} \langle \partial^2_t \PV_tg, t(\vz\phi_\epsilon)^2 (\PV_tg-c_0)\rangle \d \mu\d t\nonumber\\
&\ \ -\int_0^{2r_B}\int_{2B} V \langle \PV_tg, t(\vz\phi_\epsilon)^2(\PV_tg-c_0) \rangle \d \mu\d t .
\end{align}

{\bf Case 1}. Suppose first $c_0\neq 0$. Then a combination of \eqref{cacc-par-1} and \eqref{cacc-par-2}
yields that
\begin{align*}
&\int_{0}^{2r_B}\int_{2B}|{t}\nabla_x \PV_tg|^2(\vz\phi_\epsilon)^2\frac{\d\mu\d t}{t}\\
&\ \le C\int_0^{2r_B}\int_{2B} \left(|t^2\partial^2_t \PV_tg||\PV_tg-c_0|+t^2|\PV_t g||\PV_tg-c_0| V \right)\frac{\d \mu\d t}{t}\\
&\ +\int_0^{2r_B}\int_{2B} \left( \frac12|t\nabla_x \PV_tg|^2(\varphi\phi_\epsilon)^2
 +{2t^2}|\nabla_x \varphi|^2 \phi_\epsilon^2|\PV_tg-c_0|^2\right)\frac{\d \mu\d t}{t},
\end{align*}
and hence, as $|\nabla_x \varphi|\le C/r_B$,
\begin{align*}
\int_{2\epsilon}^{r_B}\int_{B}|{t}\nabla_x \PV_tg|^2\d\mu\frac{\d t}{t}
&\le \int_0^{2r_B}\int_{2B}|{t}\nabla_x \PV_tg|^2(\vz\phi_\epsilon)^2{\d\mu}\frac{\d t}{t}\\
&\le C\int_0^{2r_B}\int_{2B} \left(|t^2\partial^2_t \PV_tg||\PV_tg-c_0|+t^2|\PV_t g||\PV_tg-c_0| V \right)\frac{\d \mu\d t}{t}\\
&\ +C\int_0^{2r_B}\int_{2B} \left(\frac{t^2}{r_B^2}|\PV_tg-c_0|^2\right)\frac{\d \mu\d t}{t}.
\end{align*}
The estimate \eqref{key-est-2} follows by letting $\epsilon\to 0$.

{\bf Case 2}. Let $c_0=0$.  Note that as $V\ge 0$,
\begin{align*}
\int_0^{2r_B}\int_{2B} \langle \nabla_x \PV_tg, \nabla_x(t(\vz\phi_\epsilon)^2 \PV_tg)\rangle \d \mu\d t
&= \int_0^{2r_B}\int_{2B} \langle \partial^2_t \PV_tg, t(\vz\phi_\epsilon)^2 \PV_tg\rangle \d \mu\d t\\
&\ -\int_0^{2r_B}\int_{2B} t(\PV_t g)^2 (\vz\phi_\epsilon)^2V \d \mu\d t \\
&\le  \int_0^{2r_B}\int_{2B} \langle \partial^2_t \PV_tg, t(\vz\phi_\epsilon)^2 \PV_tg\rangle \d \mu\d t.
\end{align*}
We therefore deduce from this and \eqref{cacc-par-1} that
\begin{align*}
\int_{0}^{2r_B}\int_{2B}|{t}\nabla_x \PV_tg|^2(\vz\phi_\epsilon)^2\d\mu\frac{\d t}{t}
&\le \int_0^{2r_B}\int_{2B} \langle \partial^2_t \PV_tg, t(\vz\phi_\epsilon)^2 \PV_tg\rangle \d \mu\d t \\
&\ +\int_0^{2r_B}\int_{2B}   |\nabla_x \PV_tg| \frac{Ct}{r_B}\vz\phi_\epsilon^2 |\PV_tg| \d \mu\d t\\
&{\le C}\int_0^{2r_B}\int_{2B} \left(|t^2\partial^2_t \PV_tg||\PV_tg|\right)\frac{\d \mu\d t}{t} \\
&\ +\int_0^{2r_B}\int_{2B}  \left(\frac 12 |t\nabla_x \PV_tg| ^2(\vz\phi_\epsilon)^2 {+2t^2|\nabla_x\vz|^2}\phi_\epsilon^2 |\PV_tg|^2\right)\frac{\d \mu\d t}{t}.
\end{align*}
Hence, one has
\begin{align*}
\int_{2\epsilon}^{r_B}\int_{B}|{t}\nabla_x \PV_tg|^2\d\mu\frac{\d t}{t}
&\le \int_{0}^{2r_B}\int_{2B}|{t}\nabla_x \PV_tg|^2(\vz\phi_\epsilon)^2\d\mu\frac{\d t}{t}\\
&\le {C}\int_0^{2r_B}\int_{2B} \left(|t^2\partial^2_t \PV_tg||\PV_tg|  +{\frac{t^2}{r_B^2}}\phi_\epsilon^2 |\PV_tg|^2\right)\frac{\d \mu\d t}{t}\\
&\le C\int_0^{2r_B}\int_{2B} \left(|t^2\partial^2_t \PV_tg||\PV_tg|+\frac{t^2}{r_B^2}  |\PV_tg|^2 \right)\frac{\d\mu\d t}{t}.
\end{align*}
Letting $\epsilon\to 0$, we see the estimate \eqref{key-est} holds.
\end{proof}

We can now finish the proof of the last part of the main result.
\begin{proposition}\label{bmoToHmo}
 Let $(X,d,\mu,\mathscr{E})$ be a complete Dirichlet metric space satisfying $(D)$
  and admitting an $L^2$-Poincar\'{e} inequality $(P_2)$.
 Suppose that $V\in RH_q(X)\cap A_\infty(X)$ with $q>\max\{1,Q/2\}$.
  Then if $f\in \BMO_\LV(X)$, $u(x,t)=\PV_tf(x)\in \HMO_\LV(X\times\rr_+)$.
Moreover there exists a constant $C>1$ such that
$$\|u\|_{\HMO_{\LV}} \le C\|f\|_{\BMO_\LV}.$$
  \end{proposition}
\begin{proof}
By Lemma \ref{est-time-carleson}, we are left to estimate $|{t}\nabla_x \PV_t(f)|^2\d \mu\frac{\d t}{t}$. We set
$f_1:=(f-f_{4B})\mathbbm{1}_{4B}$, $f_2:=(f-f_{4B})\mathbbm{1}_{X\setminus 4B}$
and $f_3:=f_{4B}$.

{\bf Step 1}.
For the term $f_1$, it follows from the $L^2$-boundedness of the Riesz operator $\nabla_x\LV^{-1/2}$ that
\begin{eqnarray*}
\int_0^{r_B}\int_B|{t}\nabla_x \PV_t(f_1)|^2\d \mu\frac{\d t}{t}
&&\le {C}\int_0^{\infty}\int_{X}|{t}\sqrt \LV \PV_t(f_1)|^2\d \mu\frac{\d t}{t}\\
&&\le {C}\|f_1\|_{L^2(X)}\le C\mu(B)\,\|f\|^2_{\BMO_\LV}.
\end{eqnarray*}

{\bf Step 2.} Let us estimate $f_2$ in this step. By  \eqref{est-poisson-1} from Proposition \ref{est-poisson-kernel}, similar to \eqref{poisson-bmo} and \eqref{est-bmo-carleson-2v}, we conclude that for any $x\in 2B$ and $0<t<2r_B$
\begin{eqnarray*}
|t^2\partial^2_t \PV_t(f_2)(x)|&&\le C\int_{X\setminus 4B}\frac{t|f(y)-f_{4B}|}{(t+d(x,y))\mu(B(x,t+d(x,y)))}\d \mu(y)\le  Ct  r_B^{-1}\|f\|_{\BMO_\LV},
\end{eqnarray*}
and similarly,
\begin{eqnarray*}
|\PV_t(f_2)(x)|&&\le C\int_{X\setminus 4B}\frac{t|f(y)-f_{4B}|}{(t+d(x,y))\mu(B(x,t+d(x,y)))}\d \mu(y) \le Ct r_B^{-1}\|f\|_{\BMO_\LV}.
\end{eqnarray*}

Either one of the above two estimates also show that $f_2$ satisfies the requirement from Proposition
\ref{key-cacc}. We then apply {the above two estimates} to conclude that
\begin{eqnarray*}
\int_{0}^{r_B}\int_{B}|{t}\nabla_x \PV_tf_2|^2\d\mu\frac{\d t}{t}
 &&\le C\int_0^{2r_B}\int_{2B} \left(|t^2\partial^2_t\PV_t{f_2}||\PV_t{f_2}|+\frac{t^2}{r_B^2} |\PV_t{f_2}|^2 \right)\frac{\d \mu\d t}{t}\\
 && \le C\int_0^{2r_B}\int_{2B} \left(\frac{t^2}{r_B^{2}}\|f\|_{\BMO_\LV}^2 +\frac{t^3}{r_B^{3}}\|f\|_{\BMO_\LV}^2  \right)\frac{\d \mu\d t}{t}\\
&& \le C\mu(B)\|f\|_{\BMO_\LV}^2.
\end{eqnarray*}

{\bf Step 3.} In this step, we deal with $f_3$ at the case  $4r_B\ge \rho(x_B)$.
 By using \eqref{est-poisson-3},
 it holds for any $x\in 2B$, $0<t<2r_B$ and any  $N>0$ that
\begin{eqnarray*}
|t^2\partial^2_t \PV_t(f_{4B})(x)|\le C \left(\frac{t}{\rho(x)}\right)^\delta \left(1+\frac{t}{\rho(x)}\right)^{-N}|f_{4B}|\le  C\left(\frac{t}{\rho(x)}\right)^\delta \left(1+\frac{t}{\rho(x)}\right)^{-N}\|f\|_{\BMO_\LV},
\end{eqnarray*}
{where $\delta\in(0,\min\{1,2-Q/q\})$}, and
\begin{eqnarray*}
|\PV_t(f_{4B})(x)|\le C|f_{4B}|\le C\|f\|_{\BMO_\LV}.
\end{eqnarray*}
As $4r_B\ge \rho(x_B)$, by using the estimate \eqref{key-est} from Proposition \ref{key-cacc}, we repeat the argument in the proof of \eqref{est-bmo-carleson-4}
to conclude that
\begin{eqnarray*}
\int_{0}^{r_B}\int_{B}|{t}\nabla_x \PV_tf_{4B}|^2\d\mu\frac{\d t}{t}
&&\le C\int_0^{2r_B}\int_{2B} \left(|t^2\partial^2_t\PV_t{f_{4B}}||\PV_t{f_{4B}}|+\frac{t^2}{r_B^2} |\PV_t{f_{4B}}|^2 \right)\frac{\d \mu\d t}{t}\\
&&\le C\int_0^{2r_B}\int_{2B} \left[\left(\frac{t}{\rho(x)}\right)^\delta \left(1+\frac{t}{\rho(x)}\right)^{-N}{\|f\|_{\BMO_\LV}^2}
 +\frac{t^2}{r_B^2}\|f\|_{\BMO_\LV}^2\right]\frac{\d \mu(x)\d t}{t}\\
  &&\le C\mu(B)\|f\|_{\BMO_\LV}^2.
\end{eqnarray*}

{\bf Step 4.} In the last step, we treat $f_3$ for the case  $4r_B<\rho(x_B)$.

Note that  $4r_B<\rho(x_B)$,  it holds $\rho(x)\sim\rho(x_B)$ for any $x\in 2B$.
By using \eqref{est-poisson-3} from Proposition \ref{est-poisson-kernel}, it holds
for any $x\in 2B$, $0<t<2r_B$ and $N>0$ that
\begin{eqnarray}\label{est-small-1}
|t^2\partial^2_t \PV_tf_{4B}(x)|&&\le C\left(\frac{t}{\rho(x_B)}\right)^\delta |f_{4B}|
\le  C\left(\frac{t}{\rho(x_B)}\right)^\delta \left(|f_{4B}-f_{B(x_B,\rho(x_B))}|+|f_{B(x_B,\rho(x_B))}|\right)\nonumber \\
 &&\le C\left(\frac{t}{\rho(x_B)}\right)^\delta \left(1+\log\left(\frac{\rho(x_B)}{r_B}\right)\right)\|f\|_{\BMO_\LV},
\end{eqnarray}
{where $\delta\in(0,\min\{1,2-Q/q\})$}, and
\begin{eqnarray}\label{est-small-2}
|\PV_tf_{4B}(x)|\le C|f_{4B}|\le C\left(1+\log\left(\frac{\rho(x_B)}{r_B}\right)\right)\|f\|_{\BMO_\LV}.
\end{eqnarray}
Since $e^{-t\sqrt{\L}}1\equiv1$ for all $t$, we have via Bochner's subordination formula that
\begin{eqnarray*}
|\PV_tf_{4B}(x)-f_{4B}|&&=|f_{4B}||e^{-t\sqrt \LV}(1)(x)-{e^{-t\sqrt {\L}}}(1)(x)|\\
&&\le \frac{|f_{4B}|}{2\sqrt{\pi}}\int_{X}\int_0^\fz\frac{t}{s^{1/2}}
\exp\lf\{-\frac{t^2}{4s}\r\}|h^v_s(x,y)-h_s(x,y)|\frac{\d s}{s}\d \mu(y),
\end{eqnarray*}
where $h^v$, $h$ are the kernels of $e^{-t\LV}$ and $e^{-t\L}$, respectively.
By Proposition \ref{heat-kernel-diff} and the fact that $\delta\in(0,\min\{1,2-Q/q\})$,
it holds for any $x\in X$,
\begin{eqnarray*}
|\PV_tf_{4B}(x)-f_{4B}|&&\le C|f_{4B}|\int_{X}\int_0^\fz\frac{t}{s^{1/2}}\exp\lf\{-\frac{t^2}{4s}\r\}\left(\frac{\sqrt s}{\sqrt s+\rho(x)}\right)^{2-Q/q} \frac{\exp\left\{-\frac{d(x,y)^2}{cs}\right\}}{\mu(B(x,\sqrt s))}\frac{\d s}{s}\d\mu(y) \\
&&\le C|f_{4B}|\int_0^\fz\int_{X}\frac{t}{s^{1/2}}\exp\lf\{-\frac{t^2}{4s}\r\}\left(\frac{\sqrt s}{\sqrt s+\rho(x)}\right)^{\delta} \frac{\exp\left\{-\frac{d(x,y)^2}{cs}\right\}}{\mu(B(x,\sqrt s))} \d\mu(y)\frac{\d s}{s} \\
&&\le C|f_{4B}|\left(\int_0^{t^2}\frac{t}{s^{1/2}}\exp\lf\{-\frac{t^2}{4s}\r\}\left(\frac{\sqrt s}{\rho(x)}\right)^{\delta} \frac{\d s}{s}+\int_{t^2}^\infty\frac{t}{s^{1/2}}\exp\lf\{-\frac{t^2}{4s}\r\}\left(\frac{\sqrt s}{\rho(x)}\right)^{\delta}\frac{\d s}{s}\right) \\
&&\le C|f_{4B}|\left(\int_0^{t^2}\frac{t}{s^{1/2}}\frac{4s}{t^2}\left(\frac{\sqrt s}{\rho(s)}\right)^{\delta} \frac{\d s}{s}
 +\int_{t^2}^\infty\frac{t}{s^{1/2}}\left(\frac{\sqrt s}{\rho(x)}\right)^{\delta}\frac{\d s}{s}\right) \\
&&\le C|f_{4B}|\left(\frac{t}{\rho(x)}\right)^{\delta}.
\end{eqnarray*}
Above in the last inequality the integral over $(t^2,\infty)$
is convergent since $\dz<1$.
Noting that as $4r_B<\rho(x_B)$, $\rho(x)\sim\rho(x_B)$ for any $x\in 2B$, we conclude that
\begin{eqnarray}\label{est-small-3}
|\PV_tf_{4B}(x)-f_{4B}|&&\le C\left(\frac{t}{\rho(x)}\right)^{\delta} |f_{4B}|\le C\left(\frac{t}{\rho(x_B)}\right)^{\delta} \left(1+\log\left(\frac{\rho(x_B)}{r_B}\right)\right)\|f\|_{\BMO_\LV}.
\end{eqnarray}
Combining the estimates \eqref{est-small-1}, \eqref{est-small-2},
 \eqref{est-small-3}, together with the estimate \eqref{key-est-2} in Proposition \ref{key-cacc}, we arrive at
\begin{align*}
&\int_{0}^{r_B}\int_{B}|{t}\nabla_x \PV_tf_{4B}|^2\d\mu\frac{\d t}{t}\\
&\ \le C\int_0^{2r_B}\int_{2B} \left(|t^2\partial^2_t \PV_tf_{4B}||\PV_tf_{4B}-f_{4B}|+\frac{t^2}{r_B^2}  |\PV_tf_{4B}-f_{4B}|^2 \right)\frac{\d \mu\d t}{t}\\
&\ \ +\int_0^{2r_B}\int_{2B} t|\PV_t f_{4B}||\PV_tf_{4B}-f_{4B}| V \d \mu\d t\\
&\ \le C\int_0^{2r_B}\int_{2B}  \left(\frac{t}{\rho(x_B)}\right)^{2\delta} \left(1+\log\left(\frac{\rho(x_B)}{r_B}\right)\right)^2\|f\|_{\BMO_\LV}^2  \frac{\d \mu\d t}{t}\\
&\ \ + C\int_0^{2r_B}\int_{2B}  \frac{t^2}{r_B^2} \left(\frac{t}{\rho(x_B)}\right)^{2\delta}
 \left(1+\log\left(\frac{\rho(x_B)}{r_B}\right)\right)^2\|f\|_{\BMO_\LV}^2  \frac{\d \mu\d t}{t}\\
&\ \ +C\int_0^{2r_B}\int_{2B} t \left(\frac{t}{\rho(x_B)}\right)^{\delta} \left(1+\log\left(\frac{\rho(x_B)}{r_B}\right)\right)^2\|f\|_{\BMO_\LV}^2 V \d \mu\d t\\
&\ \le C\|f\|_{\BMO_\LV}^2 \left(1+\log\left(\frac{\rho(x_B)}{r_B}\right)\right)^2\left[\left(\frac{r_B}
{\rho(x_B)}\right)^{2\delta}\mu(B) +  r_B^2 \left(\frac{r_B}{\rho(x_B)}\right)^{\delta}\int_BV\d\mu \right]\\
&\ \le C\|f\|_{\BMO_\LV}^2 \left(1+\log\left(\frac{\rho(x_B)}{r_B}\right)\right)^2\left[\left(\frac{r_B}
{\rho(x_B)}\right)^{2\delta}\mu(B) +  \left(\frac{r_B}{\rho(x_B)}\right)^{\delta}\mu(B) \right]\\
&\ \le C\mu(B)\|f\|_{\BMO_\LV}^2 \left(1+\log\left(\frac{\rho(x_B)}{r_B}\right)\right)^2\left(\frac{r_B}{\rho(x_B)}\right)^{\delta}\\
&\ \le C\mu(B)\|f\|_{\BMO_\LV}^2.
\end{align*}
Above, in the last third inequality, we used the fact $\int_BV{\d\mu} \le C\mu(B)r_B^{-2}$, as $4r_B<\rho(x_B)$.

Finally, we conclude from the four steps that,
\begin{eqnarray*}
\left(\int_{0}^{r_B}\int_{B}|{t}\nabla_x \PV_tf|^2\d\mu\frac{\d t}{t}\right)^{1/2}&&\le \sum_{k=1}^3
\left(\int_{0}^{r_B}\int_{B}|{t}\nabla_x \PV_tf_k|^2\d\mu\frac{\d t}{t}\right)^{1/2} \le C\mu(B)^{1/2}\|f\|_{\BMO_\LV},
\end{eqnarray*}
which together with Lemma \ref{est-time-carleson} implies
$$\|\PV_tf\|_{\HMO_\LV}=\sup_{B(x_B,r_B)}\left(\int_{0}^{r_B}\fint_{B}|{t}(\partial_t,\nabla_x) \PV_tf|^2\d\mu\frac{\d t}{t}\right)^{1/2}\le C\|f\|_{\BMO_\LV}.$$
This completes the proof.
\end{proof}

We can now finish the proof of Theorem \ref{thm2}.
\begin{proof}[Proof of Theorem \ref{thm2}]
Since $V\in RH_{q}(X)$ with $q\ge (Q+1)/2$, there exists $\epsilon>0$
such that $V\in RH_{q+\epsilon}(X)$.

Part (i) follows from Proposition \ref{HmotoBmo}, while part (ii) follows from
Proposition \ref{bmoToHmo}.
\end{proof}

\begin{remark}\rm
One may wonder that if it is possible to relax the requirement $q\ge (Q+1)/2$
to $q>1$ together with $q\ge Q/2$ in Theorem \ref{thm2}.
For part (ii) of Theorem \ref{thm2}, this is ensured by Proposition \ref{bmoToHmo}.
However, the requirement $q\ge (Q+1)/2$ seems essential in  part (i).
In fact, for a solution of the Schr\"odinger equation $\LV u-\partial_t^2 u=0$ on $X\times\rr_+$,
$V\in RH_q$ for $q\ge (Q+1)/2$  seems to be the least condition for $u$ to be continuous,
and also for the Liouville property (Lemma \ref{Liouville-upper-space}).
\end{remark}

\section*{Acknowledgments}
\addcontentsline{toc}{section}{Acknowledgments} \hskip\parindent
R. Jiang would like to thank Professor Lixin Yan for many helpful discussions during the preparation
of this work. This work is supported by NNSF of China (11922114, 11671039 \& 11771043).

\medskip

\noindent Center for Applied Mathematics, Tianjin University, Tianjin 300072, P. R. China \smallskip \\
\noindent{E-mail}: \texttt{rejiang@tju.edu.cn}\ \& \texttt{bli.math@outlook.com}

\end{document}